\documentclass[english]{article}

\usepackage{amsmath,amsfonts,amsthm,amssymb}
\usepackage{graphicx}
\usepackage{epstopdf}
\usepackage{enumerate}
\usepackage{booktabs}
\usepackage{verbatim}
\usepackage{nicefrac}
\usepackage{ulem}

\usepackage{wasysym}
\usepackage{hyperref}
\usepackage{stmaryrd} % for double brackets \rrbracket and \llbracket

\usepackage{captdef}
\usepackage{fancyhdr}
\usepackage{tgschola}

\usepackage[overload]{empheq}
\usepackage{xcolor}
\usepackage{float}

\usepackage{lineno}

\graphicspath{{./fig/}}

\newtheorem{theorem}{Theorem}
\newtheorem{proposition}{Proposition}

\newtheorem{remark}{Remark}

\usepackage{xcolor}

\title{Analysis and numerical simulation of a spatio-temporal Ricker-type model for the control of \textit{Aedes aegypti} mosquitoes with Sterile Insect Techniques}
\vspace{1cm}

\author{ Oscar Eduardo Escobar-Lasso$^{1}$, \
 Stefan Frei$^{2}$\footnote{Corresponding author: stefan.frei@uni-konstanz.de}, \ 
 Reinhard Racke$^{2}$, 
 Olga Vasilieva$^{1}$\vspace{5mm} \\
$^1$ \small Department of Mathematics, Universidad del Valle, 760031 Cali, Colombia \\
$^2$ \small Department of Mathematics and Statistics,
University of Konstanz, 78457 Konstanz, Germany \\
}

\date{\today}

\begin{document}
\maketitle

\begin{abstract}
{Sterile Insect Technique (SIT) is widely regarded as a promising, environmentally friendly and chemical-free strategy for the prevention and control of dengue and other vector-borne diseases. 
In this paper, we develop and analyze a spatio-temporal reaction–diffusion model describing the dynamics of three mosquito subpopulations involved in SIT-based biological control of \textit{Aedes aegypti} mosquitoes. } Our sex-structured model explicitly incorporates fertile females together with fertile and sterile males that compete for mating. {Its key features include spatial mosquito dispersal and the incorporation of spatially heterogeneous external releases of sterile individuals. }

{We establish the existence and uniqueness of global, non negative, and bounded solutions, guaranteeing the mathematical well-posedness and biological consistency of the system.} A fully discrete numerical scheme based on the finite element method and an implicit-explicit time-stepping scheme is proposed {and analyzed}. {Numerical simulations confirm the presence of a critical release-size threshold governing eradication versus persistence at a stable equilibrium with reduced total population size, in agreement with the underlying ODE dynamics.} Moreover, the spatial structure of the model allows us to analyze the impact of spatial distributions, heterogeneous releases, and periodic impulsive control strategies, providing insight into the optimal spatial and temporal deployment of SIT-based interventions.
\end{abstract}

%\tableofcontents

%\newpage
\baselineskip 5mm

\section{Introduction}
\label{sec-intro}

The Sterile Insect Technique (SIT) is a species-specific biological control strategy, dating to the mid-20th century, that has been successfully applied to reduce insect populations worldwide \cite{Bourtzis2021}. The method is based on a simple but powerful principle: large numbers of mass-reared sterile males are released into the target area, where they compete with wild males for mating with wild females. Since matings with sterile males produce no viable offspring, the reproductive potential of the wild population gradually declines. If releases are sufficiently frequent and abundant, the wild population can be suppressed or even locally eliminated. 

In recent decades, SIT has been adapted for mosquito control, including \textit{Aedes aegypti} females, the primary transmitters or vectors of several vector-borne human diseases, such as dengue, Chikungunya, Zika, and yellow fever. These diseases impose substantial health and economic burdens in tropical and subtropical regions, where \textit{Aedes aegypti} mosquitoes have colonized almost all urban and peri-urban environments. For decades, conventional vector control strategies have relied heavily on chemical substances. {However, the widespread emergence of insecticide resistance, together with environmental and public health concerns, has significantly limited the long-term sustainability of chemical-based control strategies} \cite{Sarwar2015}.

On the other hand, advances in mass-rearing, sex separation, irradiation, and quality control have enabled large-scale production of sterile males, making SIT an attractive complementary or alternative tool within integrated vector management programs. Notably, SIT is an environmentally friendly approach: it does not rely on insecticide spraying, it is non-harmful to other species (being species-specific), and sterile males neither bite nor transmit pathogens to humans, making this technique completely safe for human health.  A comprehensive overview of SIT applications to \textit{Aedes aegypti} \cite{Oliva2021} provides more information regarding its operational challenges, field trials, and integration with other control strategies.

Mathematical models have become a key tool for designing and evaluating SIT-based vector control strategies because they allow researchers and decision-makers to explore release strategies, assess feasibility, and predict outcomes across different ecological and operational scenarios without relying exclusively on logistically demanding and costly field experiments.  

Models formulated using ordinary differential equations (ODEs) focus only on the temporal dynamics of \textit{Aedes aegypti} populations under SIT-based interventions in spatially homogeneous settings (see, e.g., \cite{Anguelov2020, Bliman2019, Esteva2005, Huang2017, Pang2022, Strugarek2019}). Such models are particularly useful for determining release sizes and frequencies to achieve population suppression, as well as for validating the results through practically meaningful impulsive releases. {However, ODE models inherently assume spatial homogeneity and therefore cannot capture the effects of mosquito dispersal, heterogeneous release patterns, or spatially localized suppression dynamics.}

In contrast, models relying on partial differential equations (PDEs) incorporate both temporal and spatial dynamics, accounting for mosquito dispersal across heterogeneous landscapes. {These models not only allow for the evaluation of the intensity and timing of the release, but also allow for the analysis of where sterile mosquitoes should be released to maximize the impact of SIT-based strategies; this spatial deployment geometry provides a more realistic framework for the design of such strategies.} In particular, PDE-type models provide a powerful theoretical framework for predicting mosquito spread and the spatial propagation of suppression effects across the target area, reducing the need for extensive field experimentation. 

The first PDE model incorporating spatial dispersal in the context of the Sterile Insect Technique was introduced in \cite{Manoranjan1986}. In that work, the wild insect population was assumed to follow logistic-type growth while undergoing diffusion in one spatial dimension. {Although this pioneering study did not incorporate sex or stage structure of the mosquito population, it still laid the foundations for a class of PDE models to understand how sterile releases interact with spatial dispersal processes.} Subsequent works have substantially deepened the theoretical analysis of this initial diffusion-based SIT framework. In particular, homogeneous and space-dependent continuous-time releases of sterile males on the boundary of a closed region were considered in \cite{Jiang2014, Li2012} together with a more refined stability analysis, while a more recent work \cite{Ma2024} has incorporated spatial heterogeneity directly into the mosquito's life traits by allowing their birth and mortality, as well as the release rate of sterile males to vary in space. 

It is also worth mentioning that Almeida et al. \cite{Almeida2022, Almeida2023} performed numerical tests of the model attributed to \cite{Manoranjan1986} under spatially heterogeneous releases of sterile males, demonstrating how non-uniform releases affect the suppression and spatial distribution of the wild mosquito population. A slightly different model with an Allee effect for a mosquito population distributed in several spots along a one-dimensional coastal area of a small island was studied in \cite{Multerer2019} using the optimal control approach for PDEs with unidimensional diffusion. The authors developed an optimal release strategy that imitates simultaneous or sequential releases proportional to the wild mosquito population sizes, initially concentrated along the coastline.

Sex and stage structure of \textit{Aedes aegypti} mosquitoes were also considered by some authors, resulting in more sophisticated PDE modeling frameworks in continuous time \cite{Anguelov2020p, Silva2022, Araujo2025, Dufourd2012, Dufourd2013, Leculier2023, Parshad2011, Ramirez2022}. Among these works, only two formally addressed the existence of solutions of the PDE system with two-dimensional diffusion (see \cite{Araujo2025, Parshad2011}), and the other two proved the existence of traveling wave-type solutions under one-dimensional diffusion \cite{Anguelov2020p, Leculier2023}. While several authors focused only on continuous releases of sterile males, with or without spatial variability \cite{Anguelov2020p, Silva2022, Araujo2025, Leculier2023}, others also analyzed the effects of more practice-oriented impulsive releases \cite{Dufourd2012, Dufourd2013, Ramirez2022}. 

Despite the increasing structural sophistication of spatial SIT models, almost all existing PDE formulations rely on logistic-type growth laws to describe the dynamics of the wild mosquito population, in which intraspecific competition among adult insects directly affects their mortality. While mathematically convenient, this assumption is biologically questionable for the wild \textit{Aedes aegypti}. In natural insect populations, intraspecific competition primarily occurs during the pre-adult (or aquatic) stages and regulates survival to adulthood rather than directly increasing adult mortality. {In contrast, we adopt a continuous-time Ricker-type population dynamics model in the present work, in which density dependence regulates the transition to adulthood rather than directly increasing adult mortality, providing a biologically more realistic description of mosquito population dynamics.} \cite{Bliman2019, Bliman2023, Cardona2025, Escobar2021}. To the best of our knowledge, a {spatiotemporal} SIT model for \textit{Aedes aegypti} combining Ricker-type recruitment, adult sex structure, and two-dimensional diffusion has not been previously analyzed.

From an analytical standpoint, we rigorously establish the well-posedness of the resulting PDE system by proving the existence, uniqueness, nonnegativity, and boundedness of its solutions under biologically meaningful assumptions. On the numerical side, we develop and {analyze a numerical framework} based on the finite element method coupled with an implicit–explicit time-stepping scheme, allowing us to simulate spatially heterogeneous constant and impulsive releases of sterile males. {In particular, we analyze the influence of the additional numerical diffusion inherent to different variants of the explicit-implicit time-stepping both analytically and numerically.} A rigorous numerical analysis of discretization schemes for SIT models with two-dimensional diffusion and nonlinear coupling between wild and sterile populations remains scarce. To date, only the recent study by de Araujo \textit{et al.} \cite{Araujo2025} has thoroughly addressed the development and analysis of such algorithms in this context. Moreover, impulsive releases — which more faithfully represent operational SIT programs — have been comparatively less studied in {spatiotemporal} continuous-time frameworks. {The present paper addresses these gaps by combining rigorous analytical results with a structure-aware numerical framework tailored to spatial heterogeneity, constant and impulsive control strategies, and their impact on suppression and local elimination outcomes.}

The article outline is as follows. In Section~\ref{sec-model}, we introduce the spatiotemporal model based on Ricker-type population dynamics augmented with two-dimensional diffusion. In Section~\ref{sec.analysis}, we establish the well-posedness of the model  via Lipschitz continuity and a parabolic positivity property.  Section~\ref{sec.disc} is devoted to the numerical discretization, as well as an error estimate of the discretization error due to temporal discretization. Section~\ref{sec.res} encompasses detailed numerical results, including a numerical investigation of the required release threshold, the influence of the location of the releases, and the capacity of uniform and spatially heterogeneous continuous and impulsive releases to reduce or eventually eliminate the wild mosquitoes. Finally, Section~\ref{sec-conc} summarizes our findings and highlights their practical applicability.

\section{A generalized SIT-based model}
\label{sec-model}

{In this section, we develop a {spatiotemporal} Ricker-type, sex-structured {reaction–diffusion model that incorporates adult mosquito dispersal and competitive interactions between fertile and sterile males under SIT interventions.}. The original ODE system introduced in \cite{Bliman2019} described the temporal dynamics of adult wild males, wild females, and sterile males under SIT interventions, incorporating density-dependent recruitment through a Ricker-type law while assuming spatial homogeneity. To account for mosquito dispersal and spatial heterogeneity, we extend this framework to a reaction–diffusion system posed on a bounded two-dimensional domain. The resulting model preserves the nonlinear recruitment and mating structure of the original formulation while coupling it with diffusion-driven movement of adult insects.}

We begin by assuming that three mosquito populations---wild males $M(t)$, wild females $F(t)$, and sterile males $M_S(t)$---are present at day $t \geq 0$ in a particular locality. {The ODE model proposed by} Bliman \textit{et al.} \cite{Bliman2019} {is formulated as} 

\begin{equation}\label{ODESys}
\left\{
\begin{array}{rll}
\dfrac{dM}{dt}=&r\rho\dfrac{FM}{M+\gamma M_S}e^{-\sigma\left(M+F\right)}-\mu_MM,\\[4mm]
\dfrac{dF}{dt}=&(1-r)\rho\dfrac{FM}{M+\gamma M_S}e^{-\sigma\left(M+F\right)}-\mu_FF,\\[3mm]
\dfrac{dM_S}{dt}=&\Lambda-\mu_{S}M_S,
\end{array}
\right.
\end{equation}
with initial conditions:
\begin{equation}
M(0)=M_0,\quad F(0)=F_0,\quad M_S(0)=M_{S_0}.
\end{equation}
All parameters  of the model \eqref{ODESys} are positive, and  its concise description is provided in the sequel.
First, the dimensionless sex ratio $(r\in [0,1])$ is the fraction of male offspring, and $\rho$ expresses the number of eggs produced by one female on average per day ($\text{day}^{-1}$). Hence, $r\rho F$ and $(1 - r)\rho F$ correspond to the birth rates of fertile males and fertile females, respectively, in the absence of intra-species competition in aquatic stages. The latter is modeled by the exponential term $e^{- \sigma (M+F)}$, which is consistent with Ricker-type population dynamics models. Here, the parameter $\sigma$ (individual$^{-1}$) is related to larval competition under density dependence.

Second, the ratio $\dfrac{M}{M + \gamma M_S}$ describes the  probability of successful mating after which a female is fertilized. Here, $0 <\gamma \leq 1$ measures the relative mating efficiency of sterile males compared to the wild ones. Thus, the  positive recruitment terms in the first two equations of \eqref{ODESys} stand for the effective reproduction rate of wild male and female \textit{Aedes aei} mosquitoes, i.e., the number of new individuals emerging in one day on average. Notably, sterile males produce no descendants, and their external inflow is modeled by the quantity $\Lambda \geq 0$ that expresses the number of sterile males released per day.  

Finally, all three populations $M, F,$ and $M_S$ decrease due to natural mortality for  of each subpopulation characterized by the natural death rates $\mu_M$, $\mu_F$, and $\mu_S$, respectively ($\text{day}^{-1}$).
 
{The purpose of this paper is to incorporate spatial heterogeneity into the SIT framework by modeling mosquito populations over a bounded domain $\Omega \subset \mathbb{R}^2$ representing their habitat.} Throughout the article, we assume that the boundary $\partial \Omega$ is sufficiently smooth. {Further, all three populations  $M(x,t)$, $F(x,t)$, and $M_S(x,t)$ now evolve both in time and space, where $t \geq 0$ and $x \in \Omega$ and their interaction is governed by a system of coupled reaction–diffusion equations, where diffusion represents the random spatial dispersal of adult individuals across the habitat. The diffusion coefficients $\alpha_M$, $\alpha_F$, and $\alpha_S$ quantify the mobility of each subpopulation, and may reflect biological differences in dispersal behavior between males and females.}

Precisely, we consider the following spatio-temporal model:
\begin{equation}\label{Sys}
\left\{
\begin{array}{rll}
\dfrac{\partial M}{\partial t}=&\alpha_M\Delta M+r\rho\dfrac{FM}{M+\gamma M_S}e^{-\sigma\left(M+F\right)}-\mu_MM,&x\in\Omega,\\
\dfrac{\partial F}{\partial t}=&\alpha_F\Delta F+(1-r)\rho\dfrac{FM}{M+\gamma M_S}e^{-\sigma\left(M+F\right)}-\mu_FF,&x\in\Omega,\\
\dfrac{\partial M_S}{\partial t}=&\alpha_{S}\Delta M_S+\Lambda-\mu_{S}M_S,&x\in\Omega,\\
%\dfrac{\partial M}{\partial \eta}=&\dfrac{\partial F}{\partial \eta}=\dfrac{\partial M_S}{\partial \eta}=0,&x\in\partial\Omega,
\end{array}
\right.
\end{equation}
subject to homogeneous Neumann boundary conditions,
\[
\dfrac{\partial M}{\partial \eta}=\dfrac{\partial F}{\partial \eta}=\dfrac{\partial M_S}{\partial \eta}=0\hspace{10pt}\text{on}\hspace{10pt}\partial\Omega,
\]
and initial conditions
\begin{equation}\label{IC}
M(x,0)=M_0(x),\; F(x,0)=F_0(x),\; M_S(x,0)=M_{S_0}(x),\hspace{10pt} x\in\Omega,
\end{equation}
where $M_0(x)>0$, $F_0(x)\geq0$, $M_{S_0} \neq 0$, and $M_{S_0}(x) \geq 0$ for all $x\in\Omega$. Note that we assume that the population of male mosquitoes $M$ is strictly positive throughout $\Omega$ at time $t=0$, while $M_S$ is assumed to be strictly positive only in some subset of $\Omega$. {These assumptions ensure that the denominators in~\eqref{Sys} remain strictly positive for all $t \geq 0$, thereby preventing singularities and preserving the well-posedness of the system} (see Section~\ref{sec.analysis}).

\section{Analysis of the model}
\label{sec.analysis}

{In this section,} {we establish the well-posedness and fundamental qualitative properties of solutions to the Initial Value Boundary Problem (IVBP) \eqref{Sys}–\eqref{IC}. Since the model describes mosquito population densities, it is desirable and later on very useful to guarantee that all state variables remain nonnegative for all $t \geq 0$, ensuring biological consistency.} {As main result of this section, we show the} existence, uniqueness,  nonnegativity, and boundedness of the solution {to} the model \eqref{Sys}-\eqref{IC}, when {the daily release size $\Lambda$} is considered constant over time. Once the local Lipschitz-continuity of the semilinear term $f$ below has been {proven}, this will follow by well-known arguments, similar to those in \cite{AbRaSc025}.\\
Let us consider the real Banach space $B =C(\Bar{\Omega})$ of continuous vector functions $u:\Bar{\Omega}\rightarrow\mathbb{R}^3$, endowed with the norm $\|u\|=\displaystyle\sum_{i=1}^n\displaystyle\sup_{x\in\Bar{\Omega}}|u_i(x)|.$ Then, the IVBP model \eqref{Sys}-\eqref{IC} can be formally rewritten in $B$ as follows:
\begin{equation}\label{Sys2}
\left\{
\begin{array}{rll}
\dfrac{\partial X(x,t)}{\partial t}=& D \Delta X(x,t)+f(X(x,t)),&(x,t)\in\Omega\times\mathbb{R}_+,\\
X(x,0)=& X_0(x)\geq0_{\mathbb{R}^3},&x\in\Omega,
\end{array}
\right.
\end{equation}
where $X=(M,F,M_S)^{T}, X_0=(M_0,F_0,M_{S_0})^{T}$ with $M_{S_0}\neq 0$ and $^T$ denoting a column vector, $D\in\mathbb{R}^{3\times 3}$ is the diagonal matrix $D=\text{diag}\left(\alpha_M,\alpha_F,\alpha_S\right)$, $\Delta X=(\Delta M,\Delta F,\Delta M_S)^{T}$ and $f=(f_1,f_2,f_3)^{T}$, with
\[
\left\{
\begin{array}{rl}
f_1(X)=&r\rho\dfrac{FM}{M+\gamma M_S}e^{-\sigma\left(M+F\right)}-\mu_MM,\\
f_2(X)=&(1-r)\rho\dfrac{FM}{M+\gamma M_S}e^{-\sigma\left(M+F\right)}-\mu_FF,\\
f_3(X)=&\Lambda-\mu_{S}M_S.
\end{array}
\right.
\]
{Due to fact that the quotient $\frac{M}{M+\gamma M_s}$ is bounded above by one (for nonnegative $M$ and $M_S$) the system is well-defined, as we will discuss in the following.}
{Solutions generate a continuous semiflow in appropriate function spaces (e.g., $C^k(\overline{\Omega})$ or $H^k(\Omega)$) or the Sobolev space $H^k(\Omega)=W^{k,2}(\Omega)$, and become smooth for $t>0$ due to the parabolic character of the system,} see \cite{Capasso1993, Henry1981, Mora1983, Smoller1983}. 

Before studying the main result, we will show two properties related to the model.
We decompose $f$ in \eqref{Sys2} as
\[
f(X)=(r\rho G(X),(1-r)\rho G(X),0)^T+(0,0,\Lambda)^T-(\mu_M M,\mu_F F,\mu_S M_S)^T,
    \]
where 
\begin{align}\label{defG}
G(X):
=\dfrac{FM}{M+\gamma M_S}e^{-\sigma\left(M+F\right)}.
\end{align}
The following proposition shows the Lipschitz property of the model. 

\begin{proposition}\label{PropLipschitz}
For each $X=(M,F,M_S)$ and $\widehat{X}=(\widehat{M},\widehat{F},\hat{M}_S)$ with $M,F,M_S, \widehat{M},\widehat{F},\hat{M}_S \geq 0$ and $M+\gamma M_S >0$, we have
\[
\left|G(X)-G(\widehat{X})\right|\leq\dfrac{L(X)}{M+\gamma M_S}\left|X-\widehat{X}\right|,
\] 
where
\begin{align*}
L(X)&=\sigma FM+F+\gamma^2 (F+M_S) + 1.
\end{align*}
Moreover, we have for $f$ 
\begin{align*}
\left|f(X)-f(\widehat{X})\right|\leq &\left(\max\{r, 1-r\} \dfrac{\rho L(X)}{M+\gamma M_S} + \max\{\mu_M, \mu_F, \mu_S\} \right) \left|X-\widehat{X}\right|.
\end{align*}
\end{proposition}
\begin{proof}
For $M,F,M_S,\widehat{M},\widehat{F},\hat{M}_S\geq0,$ we have
\begin{align}
\left|G(X)-G(\widehat{X})\right|&=\left|\dfrac{FM}{M+\gamma M_S}e^{-\sigma\left(M+F\right)}-\dfrac{\widehat{F}\widehat{M}}{\widehat{M}+\gamma \widehat{M}_S}e^{-\sigma\left(\widehat{M}+\widehat{F}\right)}\right| \notag \\[3mm]
&\leq\left|\dfrac{FM}{M+\gamma M_S}\right|\left|e^{-\sigma\left(M+F\right)}-e^{-\sigma\left(\widehat{M}+\widehat{F}\right)}\right|+\left|\dfrac{FM}{M+\gamma M_S}-\dfrac{\widehat{F}\widehat{M}}{\widehat{M}+\gamma \widehat{M}_S}\right|e^{-\sigma\left(\widehat{M}+\widehat{F}\right)} \notag \\[3mm]
&\leq\left|\dfrac{FM}{M+\gamma M_S}\right|\left|e^{-\sigma\left(M+F\right)}-e^{-\sigma\left(\widehat{M}+\widehat{F}\right)}\right|+\left|\dfrac{FM}{M+\gamma M_S}-\dfrac{\widehat{F}\widehat{M}}{\widehat{M}+\gamma \widehat{M}_S}\right|. \label{eq:sys1}
\end{align}
For the first term in~\eqref{eq:sys1}, we have by using estimates of the form $\big|g(x) - g(\hat{x})\big| \leq \max \limits_{x\geq 0} |g'(x)| |x-\hat{x}|$ for certain functions $g$, as well as $|\exp(-\sigma M)|, \big|\exp(-\sigma \hat{F}) \big|\leq 1$:
\[\begin{array}{rl}
\left|e^{-\sigma\left(M+F\right)}-e^{-\sigma\left(\widehat{M}+\widehat{F}\right)}\right|
\leq&\left|e^{-\sigma\left(M+F\right)}-e^{-\sigma\left(M+\widehat{F}\right)}\right|+\left|e^{-\sigma\left(M+\widehat{F}\right)}-e^{-\sigma\left(\widehat{M}+\widehat{F}\right)}\right|\\[2mm]
\leq&e^{-\sigma M}\left|e^{-\sigma F}-e^{-\sigma\widehat{F}}\right|+{e^{-\sigma\widehat{F}}}\left|e^{-\sigma M}-e^{-\sigma\widehat{M}}\right|\\[2mm]
\leq&e^{-\sigma M}\max \limits_{F\geq0}\left|-\sigma e^{-\sigma F}\right|\left|F-\widehat{F}\right|+e^{-\sigma\widehat{F}}\max \limits_{M\geq0}\left|-\sigma e^{-\sigma M}\right|\left|M-\widehat{M}\right|\\[3mm]
\leq&\sigma\left|X-\widehat{X}\right|.
\end{array}\]
For the second term in~\eqref{eq:sys1}, we have
\begin{align}\label{secondrhs}
\left|\dfrac{FM}{M+\gamma M_S}-\dfrac{\widehat{F}\widehat{M}}{\widehat{M}+\gamma \widehat{M}_S}\right|
\leq &\left| \dfrac{FM\widehat{M}-\widehat{F}\widehat{M}M}{(M+\gamma M_S)(\widehat{M}+\gamma \widehat{M}_S)} \right| + \gamma\left|\dfrac{FM\hat{M}_S-\widehat{F}\widehat{M}M_S}{(M+\gamma M_S)(\widehat{M}+\gamma \widehat{M}_S)}\right|
\end{align}
The first term on the right-hand side of~\eqref{secondrhs} is bounded by
\begin{align*}
    \left| \dfrac{FM\widehat{M}-\widehat{F}\widehat{M}M}{(M+\gamma M_S)(\widehat{M}+\gamma \widehat{M}_S)} \right| \leq \left| \dfrac{M}{M+\gamma M_S} \right|  \left| \dfrac{\widehat{M}}{\widehat{M}+\gamma \widehat{M}_S}\right| \left| F-\widehat{F}\right| \leq \left| F-\widehat{F}\right|{\leq\left|X-\widehat{X}\right|}. 
\end{align*}
The second term on the right-hand side of~\eqref{secondrhs} can be estimated as follows 
\begin{align*}
\gamma\left|\dfrac{FM\hat{M}_S-\widehat{F}\widehat{M}M_S}{(M+\gamma M_S)(\widehat{M}+\gamma \widehat{M}_S)}\right|
\leq & \ \gamma \left(\left|\dfrac{F\hat{M}_S(M-\widehat{M})}{(M+\gamma M_S)(\widehat{M}+\gamma \widehat{M}_S)}\right|+\left|\dfrac{\widehat{M}F(\widehat{M}_S-M_S)}{(M+\gamma M_S)(\widehat{M}+\gamma \widehat{M}_S)}\right|\right.\\[2mm]
& \hspace{1cm} \left.+\left|\dfrac{\widehat{M}M_S(F-\widehat{F})}{(M+\gamma M_S)(\widehat{M}+\gamma \widehat{M}_S)}\right|\right)\\[2mm]
\leq &  \ \gamma \left(\dfrac{F\left|M-\widehat{M}\right|}{\gamma(M+\gamma M_S)}+\dfrac{{\gamma}F\left|\widehat{M}_S-M_S\right|}{M+\gamma M_S}+\dfrac{{\gamma}M_S\left|F-\widehat{F}\right|}{M+\gamma M_S}\right) \\[2mm]
\leq &\dfrac{1}{M+\gamma M_S}\left(F+\gamma^2 F+\gamma^2 M_S\right)\left|X-\widehat{X}\right|.
\end{align*}
Combining the three estimates, we obtain the first statement of   Proposition ~\ref{PropLipschitz}. The second estimate then follows by straight-fowards calculations from the definition of $f$.
\end{proof}

Proposition ~\ref{PropLipschitz} ensures the local Lipschitz continuity of $f$, if the denominator $M+\gamma M_S$ is positive. To guarantee this for the time-dependent system, we prove the following result for the population of sterile mosquitoes $M_S$, which is governed by a purely parabolic equation.
\begin{proposition}\label{PropLowerBounded}
Let $M_{S_0} \neq 0$ with $M_{S_0}(x) \geq 0$ for all $x\in\Omega$. Then, for every $T> \epsilon>0$ there exists a $\delta>0$ such that $M_S( \cdot,t)\geq \delta$ for all $t\in [\epsilon,T]$.
\end{proposition}
\begin{proof}
W.l.o.g. we may assume $\mu_s=0$, otherwise consider $\tilde M_S (x,t):= e^{\mu_St}M_S(x,t)$. Since $\Lambda\geq 0$ we may also assume that $\Lambda=0$, by the parabolic maximum principle. Finally we may assume 
$\alpha_S=1$, otherwise consider 
$\hat M_S(x,t):=M_S(\sqrt{\alpha_S}\,x,t)$.

For the solution of the heat equation $\frac{\partial M_S}{\partial t} - \Delta M_S=0$, studied in an arbitrary smoothly bounded domain $\Omega$ in $\mathbb{R}^n$, $n\in\mathbb{N}$, with Neumann boundary condition, we have the representation
$$
M_S(x,t) = \int \limits_{\Omega} K(x,t,y)\, M_{S_0}(y) dy,
$$
with a kernel $K$ satisfying $K\in C^\infty\left(\mathbb{R}^n \times (0,\infty) \times \mathbb{R}^n\right)$, and we conclude the following uniform estimate from \cite[Theorem 4]{ChOuYa006}:
\begin{equation} \label{lowerest}
 \exists\,c_1>0\;\; \forall\,x,y\in \overline{\Omega}\;\; \forall\; t>0:\quad
 K(x,t,y) \geq c_1\, t^{-n/2}\,e^{-c_1|x-y|^2/t}.
\end{equation}
This implies, with $d:=\max \limits_{x,y\in \overline{\Omega}} |x-y|$,
\begin{eqnarray*}
 M_S(x,t) &\geq& c_1\,t^{-n/2} \int \limits_{\Omega} e^{-c_1|x-y|^2/t} M_{S_0}(y) dy
 \\  
 &\geq& c_1\,T^{-n/2}e^{-c_1d^2/\epsilon} \int \limits_{\Omega}  M_{S_0}(y) dy =:\delta.
\end{eqnarray*}
\end{proof}

Together with the assumption \eqref{IC} including $M_0>0$, we obtain the Lipschitz continuity of $f$, and hence first a local solution. Combining the previous results, we obtain the global well-posedness of the reaction–diffusion system, which is the main result of this section. 
\begin{theorem}\label{TheoWP}
For any given initial conditions which satisfy \eqref{IC}, there exists a unique global solution of model \eqref{Sys}-\eqref{IC}, which is nonnegative and bounded.
\end{theorem}
\begin{proof}
The nonnegativity follows from {the invariance principle for parabolic systems} (see \cite[Thm. 14.14]{Smoller1983}) observing the following. Let $H_i:=-X_i$ and $\Omega_i:= \{X_i \geq 0 \}$, then the necessary condition $\nabla H_i\cdot f(X) \leq 0$ on $\partial\Omega_i$ follows easily since $\nabla H_i\cdot f(X) = -f_i\leq 0$ on $\{X_i=0\}$. 

For the boundedness, we observe
$$
|f(X)| \leq C_1 |X| + C_2,
$$
with positive constants $C_1,C_2$. Then, an application of \cite[Lemma 3.1]{Peng2012} yields
$$
\sup_{t\geq 0} \Big\|X(\cdot,t) \Big\|_{L^\infty(\Omega)} < \infty,
$$
{which allows the local solution to be extended globally in time. Consequently, the model admits a unique global solution that remains nonnegative and uniformly bounded for all $t \geq 0$.}

\end{proof}

\section{Discretization}
\label{sec.disc}

We consider the variational formulation of the IVBP  \eqref{Sys}-\eqref{IC} model, {which provides the natural functional framework for the numerical approximation of the system.} For each $t\in[0,T]$ we have for the solution $X(t)=\left(M(t),F(t),M_S(t)\right)$:
\begin{subequations}
\label{VFCP}
\begin{align}
\label{VFCP-M}
(\partial_t M(t), \varphi_M)_{\Omega}+\alpha_M\left(\nabla M(t), \nabla\varphi_M\right)_{\Omega}+\mu_M(M(t), \varphi_M)_{\Omega}-r\rho\left(G(X(t)),\varphi_M\right)_{\Omega}&=0,\\[2mm]
\label{VFCP-F}
(\partial_t F(t), \varphi_F)_{\Omega}+\alpha_F\left(\nabla F(t), \nabla\varphi_F\right)_{\Omega}+\mu_F(F(t), \varphi_F)_{\Omega}-(1-r)\rho\left(G({X(t)}),\varphi_F\right)_{\Omega}&=0,\\[2mm]
\label{VFCP-Ms}
(\partial_t M_S(t), \varphi_S)_{\Omega}+\alpha_S\left(\nabla M_S(t), \nabla\varphi_S\right)_{\Omega}+\mu_S(M_S(t), \varphi_S)_{\Omega}&=\left(\Lambda,\varphi_S\right)_{\Omega},
\end{align}
\end{subequations}
for all $\varphi_M, \varphi_F,\varphi_S\in H^1(\Omega)$ and $G(\cdot)$ defined in \eqref{defG}.\\

\subsection{Temporal Discretization}

We start with the temporal discretization of~\eqref{VFCP}. Therefore, we split the total time interval of interest $I=[0,T]$ into subintervals $I_m=(t_{m-1}, t_m], \  m=1,\ldots,N$, such that
\begin{align*}
0= t_0 < t_1 < \ldots < t_N= T.
\end{align*}
For simplicity, we consider an equidistant time grid: $\delta t_m = t_m - t_{m-1} = \delta t, \ m=1,\ldots,N$. {On each subinterval, we apply a semi-implicit first-order time-stepping scheme of IMEX type, treating the nonlinear reaction term explicitly while discretizing the diffusion and linear decay terms implicitly, see e.g.,~\cite{AscherRuuthWetton95, FreiHeinlein2023}.} For the linear part, we use a one-step-$\theta$ scheme with $\theta \in [\frac{1}{2},1]$. {The explicit treatment of the nonlinear term ensures that, at each time step, only a linear elliptic problem needs to be solved, significantly reducing the computational cost.} 

 Precisely, we approximate $X(t_m) = \big( M(t_m), F(t_m), M_S(t_m) \big)$ for $m=1,\ldots,N$ by the solution $X^m=\left(M^m, F^m, M_S^m\right)$ of the time-discrete problem, when $X^{m-1}=\left(M^{m-1}, F^{m-1}, M_S^{m-1}\right)$ is given:

\begin{eqnarray}
\begin{aligned}\label{VFDP}
\left(\dfrac{M^m-M^{m-1}}{\delta t}, \varphi_M\right)_{\Omega}+\alpha_M  \Big(\theta \nabla M^m + (1-\theta) \nabla M^{m-1}, \nabla\varphi_M\Big)_{\Omega}\hspace*{3cm}&\\
+\mu_M \Big( \theta M^m+ (1-\theta) M^{m-1}, \varphi_M \Big)_{\Omega}-r\rho \Big(G(X^{m-1}),\varphi_M \Big)_{\Omega}&=0,\\[2mm]
\left(\dfrac{F^m-F^{m-1}}{\delta t}, \varphi_F\right)_{\Omega}+\alpha_F\Big(\theta\nabla F^m + (1-\theta)\nabla F^{m-1}, \nabla\varphi_F \Big)_{\Omega}\hspace*{3cm} &\\
+\mu_F \Big(\theta F^m + (1-\theta) F^{m-1}, \varphi_F \Big)_{\Omega}-(1-r) \rho \Big(G(X^{m-1}),\varphi_F\Big)_{\Omega}&=0,\\[2mm]
\left(\dfrac{M_S^m-M_S^{m-1}}{\delta t}, \varphi_S\right)_{\Omega}+\alpha_S \Big(\theta\nabla M_S^m + (1-\theta)\nabla M_S^{m-1}, \nabla\varphi_S \Big)_{\Omega}\hspace*{3cm} &\\
+\mu_S \Big(\theta M_S^m + (1-\theta)M_S^{m-1}, \varphi_S \Big)_{\Omega}
&=\left(\Lambda,\varphi_S\right)_{\Omega}&
\end{aligned}
\end{eqnarray}
for all $\varphi_M, \varphi_F,\varphi_S\in H^1(\Omega).$ For $m=0$, we set $(M^0, F^0, M_S^0) = \big( M(0), F(0), M_S(0) \big)$. 

To estimate the time discretization error, we introduce the following notation. Let $\|\cdot\|_\Omega := \|\cdot\|_{L^2(\Omega)}$ denote the standard $L^2(\Omega)$-norm. For a function $X:=(M,F,M_S)$, we define the following $L^2$- and energy-type norms
\begin{align*}
\|X\|_{\Omega} &:= \Big(\|M\|_\Omega^2 +\|F\|_\Omega^2 + \|M_S\|_\Omega^2\Big)^{1/2} \\
\|X\|_* &:= \Big(\mu_M \|M\|_\Omega^2 + \alpha_M\|\nabla M\|_\Omega^2 +\mu_F \|F\|_\Omega^2 + \alpha_F\|\nabla F\|_\Omega^2 + \mu_S \|M_S\|_\Omega^2 + \alpha_S\|\nabla M_S\|_\Omega^2 \Big)^{1/2}.
\end{align*}
{The norm $|\cdot|_*$ corresponds to the natural energy norm associated with the linear part of the reaction–diffusion operator.}\\
For a time interval $I$ and a {Banach} space $V$, we denote by {$W^{k,\infty}(I, V)$} the corresponding Bochner space with norm {$\|\cdot\|_{W^{k,\infty}(I, V)}$, where $W^{k,\infty}$ denotes the Sobolev space with (weak) derivatives of order $k$ being bounded in $L^{\infty}$}.

\begin{theorem}\label{theo.conv}
 Let the solution $X=(M,F,M_S)$ to~\eqref{VFCP} be in $W^{1,\infty}(I, H^1(\Omega)) \cap W^{2,\infty}(I, L^2( \Omega))$ and denote the solution to~\eqref{VFDP} by $X^m:= (M^m, F^m, M_S^m)$ for $m=1,\ldots,M$ and with $\theta\in[\frac{1}{2}, 1]$. For the discretization error $e_n^X =X(t_n)-X^n$, it holds for $n\leq N$ with a constant $c>0$
\begin{align*}
\big\| e_n^X \big\|_\Omega^2 + \frac{(1-\theta)^2}{\theta} \delta t \big\| e_n^X \big\|_*^2 + \delta t \sum_{m=1}^n &\left(\frac{1}{\delta t} \big\|e_m^X - e_{m-1}^X \big\|_\Omega^2  + \left(2-\frac{1}{\theta}\right) \big\|e_m^X \big\|_*^2 + \theta \left\| e_m^X +\frac{1-\theta}{\theta} e_{m-1}^X\right\|_*^2\right)  \\ 
&\qquad\leq c\exp \big(c_\delta t_M \big) \delta t^2 \left(\big\|\partial_t X\big\|_{L^\infty([0,t_n], H^1(\Omega))}^2 +\left\|\partial_t^2 X\right\|_{L^\infty([0,t_n], L^2( \Omega))}^2\right),     
\end{align*}
where $c_\delta = \dfrac{\rho^2}{\delta_{\min}^2} \big\|L(X) \big\|_{L^\infty([0, t_n]\times\Omega)}^2$ and $\delta_{\min} := \min\limits_{(x,t)\in \Omega\times I} \big( M(x,t)+\gamma M_S(x,t) \big)>0$.
\end{theorem}
\begin{remark}\label{rem.diff}
Theorem~\ref{theo.conv} implies in particular first-order convergence ${\cal O}(\delta t)$ for $\big\|M(t_m) - M^m \big\|_\Omega, \ \big\|F(t_m) - F^m \big\|_\Omega$, and $\big\| M_S(t_m) - M_S^m \big\|_\Omega$ for $\delta t\to 0$ (after taking the square root). The term under the sum can be viewed as an additional \textit{numerical diffusion} introduced by the time discretization, and increases as $\theta$ grows. 
\end{remark}
\begin{proof}
We start by deriving an estimate for $M$ based on the first equation in~\eqref{VFDP}. We set $t=t_{m-1}$ in \eqref{VFCP-M} and subtract it from \eqref{VFDP} to get
\begin{align}
\begin{split}\label{firsterr}
&\left(\partial_tM(t_{{m-1}})-\dfrac{M(t_m)-M(t_{m-1})}{\delta t},\varphi_M\right)_{\Omega}+\left(\dfrac{M(t_m)-M(t_{m-1})}{\delta t}-\dfrac{M^m-M^{m-1}}{\delta t}, \varphi_M\right)_{\Omega}\\[2mm]
&\qquad+r\rho\Big(G(X^{m-1})-G \big(X(t_{m-1}) \big),\varphi_M\Big)_{\Omega} +\alpha_M\Big(\nabla \big(M(t_{m-1})- \theta M^m - (1-\theta)M^{m-1}\big), \nabla\varphi_M\Big)_{\Omega}\\[2mm]
&\qquad \qquad\qquad+\mu_M\Big(M(t_{m-1})-\theta M^m- (1-\theta)M^{m-1}, \varphi_M\Big)_{\Omega}=0
\end{split}
\end{align}
for all $\varphi_M\in H^1(\Omega)$. By a first-order Taylor expansion, it holds that
\begin{align*}
    \big\|M(t_m) - M(t_{m-1}) \big\|_{H^1(\Omega)} \leq \delta t \max_{t\in[t_{m-1},t_m]}\|\partial_t M\|_{H^1(\Omega)}.
\end{align*}
We apply this estimate to the {last two} terms in~\eqref{firsterr} in combination with the Cauchy-Schwarz inequality. Using the notation $e_m^M:=M(t_m)-M^m$ for $m=0,\ldots,N$, we obtain from~\eqref{firsterr}
\begin{multline}\label{seconderr}
\left(\partial_tM(t_{m-1})-\dfrac{M(t_m)-M(t_{m-1})}{\delta t},\varphi_M\right)_{\Omega}+\left(\dfrac{e_m^M-e_{m-1}^M}{\delta t}, \varphi_M\right)_{\Omega}\\[2mm]
+\alpha_M \Big(\nabla \left(\theta e_m^M - (1-\theta)e_{m-1}^M\right), \nabla\varphi_M\Big)_{\Omega}+\mu_M\Big(\theta e_m^M- (1-\theta)e_{m-1}^M, \varphi_M\Big)_{\Omega}\\[2mm]
+r\rho\Big(G \big(X^{m-1} \big)-G \big(X(t_{m-1}) \big),\varphi_M\Big)_{\Omega} \leq \delta t \big\|\varphi_M \big\|_{H^1(\Omega)}\max_{t\in[t_{m-1},t_m]} \big\|\partial_t M \big\|_{H^1(\Omega)}.
\end{multline}
Choosing $\varphi_M=e_m^M$ and using that by a telescope equality 
\begin{align*}
\Big(\theta e_m^M + (1-\theta)e_{m-1}^M,  e_m^M\Big)_\Omega &= \theta \left(e_m^M + \frac{1-\theta}{\theta} e_{m-1}^M,  e_m^M\right)_\Omega \\[2mm]
&=\frac{\theta}{2} \left( \big\|e_m^M \big\|_\Omega^2 + \left\|e_m^M +  \frac{1-\theta}{\theta} e_{m-1}^M \right\|_\Omega^2  -\left(\frac{1-\theta}{\theta}\right)^2 \big\|e_{m-1}^M \big\|_\Omega^2 \right),
\end{align*}
and similarly for the second and third term in~\eqref{seconderr}, we get
\begin{multline}\label{longest1}
\dfrac{1}{2\delta t}\left( \big\|e_m^M \big\|_\Omega^2 + \left\|e_m^M-e_{m-1}^M\right\|^2_{\Omega} -  \big\|e_{m-1}^M \big\|_\Omega^2 \right)+ \frac{\theta}{2} \alpha_M \left(\left\|\nabla e_m^M\right\|^2_{\Omega} +  \left\|\nabla \left( e_m^M +  \frac{1-\theta}{\theta} e_{m-1}^M \right) \right\|_\Omega^2 \right. \\[2mm] 
\left. -\left(\frac{1-\theta}{\theta}\right)^2 \|\nabla e_{m-1}^M\|_\Omega^2\right)
+ \frac{\theta}{2} \mu_M\left(\big\| e_m^M \big\|^2_{\Omega}   + \left\|e_m^M +  \frac{1-\theta}{\theta} e_{m-1}^M \right\|_\Omega^2 -  \big\|e_{m-1}^M \big\|_\Omega^2\right) \\[2mm]
\leq \Bigg(\delta t \max_{t\in[t_{m-1},t_m]} \big\|\partial_t M \big\|_{H^1(\Omega)} +  \left\|\partial_tM(t_m)-\dfrac{M(t_m)-M(t_{m-1})}{\delta t}\right\|_{\Omega}  \\[2mm]
+r\rho\Big\| G \big({X^{m-1}} \big)-G \big({X(t_{m-1}}) \big) \Big\|_{\Omega}\Bigg)  \big\| e_m^M \big\|_{H^1(\Omega)}, 
\end{multline}
where we have used the Cauchy-Schwarz inequality on the right-hand side.
By another Taylor expansion, for the second term on the right-hand side {of~\eqref{longest1}}, it holds that
\begin{align*}
 \left\|\partial_tM(t_{m-1})-\dfrac{M(t_m)-M(t_{m-1})}{\delta t}\right\|_{\Omega}
\leq\dfrac{\delta t}{2}\max_{t\in[t_{m-1},t_m]}\Big\|\partial_t^2M(t)\Big\|_{\Omega}.
\end{align*}
Using the Lipschitz continuity of $G(\cdot)$ shown in Proposition~\ref{PropLipschitz}, we obtain for the third term on the right-hand side of \eqref{longest1} that
\begin{align*}
\Big\|G({X^{m-1}})-G(X({t_{m-1}})) \Big\|_{\Omega} \leq \frac{L(X({t_{m-1}}))}{\delta_{\rm min}} \Big\|{X^{m-1}} - X({t_{m-1}}) \Big\|_\Omega.
\end{align*}
Using Young's inequality {with constant $c_M= \dfrac{1}{2\theta \min\{\alpha_M,\mu_M\}}$}, and absorbing the term $\left\|e_m^M\right\|_{H^1(\Omega)}$ into the left-hand side, we obtain from~\eqref{longest1}
\begin{multline}\label{longest}
\dfrac{1}{2\delta t}\left( \big\|e_m^M \big\|_\Omega^2 + \left\|e_m^M-e_{m-1}^M\right\|^2_{\Omega}\right)+ \dfrac{\theta}{2} \left(\alpha_M\left\| \nabla e_m^M\right\|^2_{\Omega} +\mu_M\left\|e_m^M\right\|^2_{\Omega} + \alpha_M \left\|\nabla \left(e_m^M +  \frac{1-\theta}{\theta} e_{m-1}^M \right) \right\|_\Omega^2 \right. \\[2mm]
\left. + \mu_M \left\|e_m^M +  \frac{1-\theta}{\theta} e_{m-1}^M \right\|_\Omega^2\right) \\[2mm]
\leq  c_M\delta t^2 \left(\max_{t\in[t_{m-1},t_m]} \big\|\partial_t M \big\|_{H^1(\Omega)}^2 + \dfrac{1}{4}\max_{t\in[t_{m-1},t_m]}\left\|\partial_t^2M(t)\right\|_{\Omega}^2\right) + \frac{1}{2\delta t } \big\|e_{m-1}^M \big\|_\Omega^2 \\[2mm]
+\frac{r^2\rho^2}{2\delta_{\min}^2} 
L(X({t_{m-1}}))^2 \Big\|{X^{m-1}} - X({t_{m-1}}) \Big\|_\Omega^2 
 + \frac{(1-\theta)^2}{2\theta} \Big(\alpha_M \big\|\nabla e_{m-1}^M \big\|_\Omega^2 + \mu_M \big\|e_{m-1}^M \big\|_\Omega^2 \Big).
\end{multline}

Analogously, we obtain for the error $e_m^F:=F(t_m)-F^m$ from the second equation in~\eqref{VFDP}
\begin{multline}
\dfrac{1}{2\delta t}\left( \big\|e_m^F \big\|_\Omega^2 + \left\|e_m^F-e_{m-1}^F\right\|^2_{\Omega}\right)+ \frac{\theta}{2} \left(\alpha_F\left\|\nabla e_m^F\right\|^2_{\Omega} +\mu_F\left\|e_m^F\right\|^2_{\Omega} + \alpha_F \left\|\nabla \left(e_m^F +  \frac{1-\theta}{\theta} e_{m-1}^F \right) \right\|_\Omega^2 \right. \\[2mm]
\left. + \mu_F \left\|e_m^F +  \frac{1-\theta}{\theta} e_{m-1}^F \right\|_\Omega^2\right) \\[2mm]
\leq  c_F\delta t^2 \left(\max_{t\in[t_{m-1},t_m]} \big\|\partial_t F \big\|_{H^1(\Omega)}^2  + \dfrac{1}{4}\max_{t\in[t_{m-1},t_m]}\left\|\partial_t^2 F(t)\right\|_{\Omega}^2\right)  + \frac{1}{2\delta t } \big\|e_{m-1}^F \big\|_\Omega^2 
\\[2mm]+\frac{(1-r)^2\rho^2}{2\delta_{\min}^2} 
L(X({t_{m-1}}))^2 \Big\|{X^{m-1}} - X({t_{m-1}}) \Big\|_\Omega^2
 + \frac{(1-\theta)^2}{2\theta}\Big(\alpha_F \big\|\nabla e_{m-1}^F \big\|_\Omega^2 + \mu_F \big\|e_{m-1}^F \big\|_\Omega^2\Big) \hspace{5mm}
\end{multline}   
with $c_F>0$ and for $e_m^{M_S}:=M_S(t_m)-M_S^m$ from the third equation
\begin{multline}
\dfrac{1}{2\delta t}\left( \left\|e_m^{M_S} \right\|_\Omega^2 + \left\|e_m^{M_S}-e_{m-1}^{M_S}\right\|^2_{\Omega}\right)+ \frac{\theta}{2} \left(\alpha_S\left\|\nabla e_m^{M_S}\right\|^2_{\Omega} +\mu_S\left\|e_m^{M_S}\right\|^2_{\Omega} + \alpha_S \left\|\nabla \left( e_m^{M_S} +  \dfrac{1-\theta}{\theta} e_{m-1}^{M_S} \right) \right\|_\Omega^2 \right. \\[2mm]
\left. + \mu_S\left\| e_m^{M_S} +  \frac{1-\theta}{\theta} e_{m-1}^{M_S} \right\|_\Omega^2\right) \\[2mm]
\leq c_{M_S}\delta t^2 \left(\max_{t\in[t_{m-1},t_m]} \left\|\partial_t M_S \right\|_{H^1(\Omega)}^2 + \dfrac{1}{4}\max_{t\in[t_{m-1},t_m]} \left\|\partial_t^2M_S(t)\right\|_{\Omega}^2\right)
+ \frac{1}{2\delta t } \left\|e_{m-1}^{M_S} \right\|_\Omega^2 \\[2mm]
+ \frac{(1-\theta)^2}{2\theta} \left(\alpha_S \left\|\nabla e_{m-1}^{M_S} \right\|_\Omega^2 + \mu_S \left\|e_{m-1}^{M_S}\right\|_\Omega^2\right), \hspace{5mm}
\end{multline}   
where $c_{M_S}>0$. Adding the three equations and using that $r^2+ (1-r)^2 \leq 1$ for $r\in [0,1]$ yields
\begin{multline}
\dfrac{1}{2\delta t}\left( \left\|e_m^X \right\|_\Omega^2 + \left\|e_m^{X}-e_{m-1}^X\right\|^2_{\Omega}\right) +\frac{\theta}{2} \left(\left\|e_m^X\right\|^2_* + \left\|e_m^X +  \frac{1-\theta}{\theta} e_{m-1}^X \right\|_*^2\right) \\[2mm]
\leq  c\delta t^2 \left(\max_{t\in[t_{m-1},t_m]} \big\|\partial_t X \big\|_{H^1(\Omega)}^2 + \dfrac{1}{4}\max_{t\in[t_{m-1},t_m]}\left\|\partial_t^2 X(t)\right\|_{\Omega}^2\right)  +\frac{\rho^2}{2\delta_{\min}^2} L\big(X({t_{m-1}}) \big)^2 \big\|e_{{m-1}}^X \big\|_{\Omega}^2\\[2mm] 
+ \frac{1}{2\delta t } \big\|e_{m-1}^X \big\|_\Omega^2 
+ \frac{(1-\theta)^2}{2\theta} \big\|e_{m-1}^X \big\|_*^2. \hspace{6mm}
\end{multline}
We split the factor in front of $\big\|e_m^X \big\|_*$ into\, $\dfrac{\theta}{2}=\dfrac{(1-\theta)^2}{2\theta} + \left(1-\dfrac{1}{2\theta}\right)$ to match the factor in front of $\big\|e_{m-1}^X\big\|_*^2$ on the right-hand side.
Summation {over} $m=1,\ldots,n$ and multiplication by $2\delta t$ results in 
\begin{multline*}
\|e_n^X\|_\Omega^2 + \frac{(1-\theta)^2}{\theta} \delta t \left\|e_n^X\right\|^2_* + \delta t \sum_{m=1}^n  \frac{1}{\delta t}\left\|e_m^{X}-e_{m-1}^X\right\|^2_{\Omega} + \left(2-\frac{1}{\theta}\right)\|e_m^X\|_* + \theta  \|e_m^X +  \frac{1-\theta}{\theta} e_{m-1}^X\|_\Omega^2\\
\leq  c\delta t^2 \left(\max_{t\in[0,t_n]}\|\partial_t X\|_{H^1(\Omega)}^2 + \dfrac{1}{4}\max_{t\in[0,t_n]}\left\|\partial_t^2 X(t)\right\|_{\Omega}^2\right)  +\left\|e_0^X\right\|^2 + \frac{(1-\theta)^2}{\theta} \delta t \left\|e_0^X\right\|^2_* \\
+ \delta t {\sum_{m=0}^{n-1}} \frac{\rho^2}{\delta_{\min}^2} L(X(t_m))^2 \|e_m^X\|_{\Omega}^2. 
\end{multline*} 

Using that, by definition $e_0^X = 0$, application of a discrete Gronwall inequality yields the statement.
\end{proof}

\subsection{Spatial Discretization by Finite Elements}

We discretize the system~\eqref{VFDP} using continuous piecewise linear finite elements in space. To this purpose, let ${\cal T}_h$ be  a triangulation of the domain $\Omega$ into triangular elements $T$ of (maximum) size $h$. The space of piecewise linear finite elements is defined as
\begin{align*}
{\cal V}_h:= \{ v\in C(\bar{\Omega}\, |\, v|_T \in P_1(T) \, \forall T\in {\cal T}_h\},  
\end{align*}
where $P_1(T)$ denotes the space of linear functions on $T$. 

Let $\big (M_h^0, F_h^0, M_{S,h}^0 \big) := \big( M(0), F(0), M_S(0) \big)$. {The resulting fully discrete finite element scheme reads as follows:} \textit{For $m=1,\ldots,M$, find $X_h^m = \big( M_h^m,F_h^m, M_{S,h}^m \big)\in \left({\cal V}_h\right)^3$, such that 
}
\begin{eqnarray}
\begin{aligned}\label{fullydisc}
\left(\dfrac{M_h^m-M_h^{m-1}}{\delta t}, \varphi_M\right)_{\Omega}+\alpha_M  \Big(\theta \nabla M_h^m + (1-\theta) \nabla M_h^{m-1}, \nabla\varphi_M\Big)_{\Omega}\hspace*{3cm}&\\[2mm]
+\mu_M \Big( \theta M_h^m+ (1-\theta) M_h^{m-1}, \varphi_M \Big)_{\Omega}-r\rho\Big(G(X_h^{m-1}),\varphi_M\Big)_{\Omega}&=0,\\[2mm]
\left(\dfrac{F_h^m-F_h^{m-1}}{\delta t}, \varphi_F\right)_{\Omega} +\alpha_F \Big( \theta\nabla F_h^m + (1-\theta)\nabla F_h^{m-1}, \nabla\varphi_F\Big)_{\Omega}\hspace*{3cm} &\\[2mm]
+\mu_F \Big(\theta F_h^m + (1-\theta) F_h^{m-1}, \varphi_F \Big)_{\Omega} - (1-r) \rho\Big(G(X_h^{m-1}),\varphi_F\Big)_{\Omega}&=0,\\[2mm]
\left(\dfrac{M_{S,h}^m-M_{S,h}^{m-1}}{\delta t}, \varphi_S\right)_{\Omega}+\alpha_S\Big(\theta\nabla M_{S,h}^m + (1-\theta)\nabla M_{S,h}^{m-1}, \nabla\varphi_S\Big)_{\Omega}
+\mu_S \Big( \theta M_{S,h}^m + (1-\theta)M_{S,h}^{m-1}, \varphi_S \Big)_{\Omega}&=\big(\Lambda,\varphi_S\big)_{\Omega}
\end{aligned}
\end{eqnarray}
for all $\big(\varphi_M, \varphi_F,\varphi_S \big)\in \left({\cal V}_h\right)^3.$\\
{At each time step, the scheme leads to three {uncoupled linear equations}, since the nonlinear reproduction term is evaluated explicitly at the previous time level. Standard finite element approximation theory combined with the temporal error estimate yields first-order convergence in time and space under suitable regularity assumptions.}

\section{Numerical results}
\label{sec.res}

In this section, we present numerical simulations to illustrate the results obtained in the previous sections, compare them with the ODE model presented in Bliman \textit{et al.} in \cite{Bliman2019} and further investigate the capabilities and limitations of the PDE model. The values of the parameters related to mosquito life traits used in the simulations are summarized in Table \ref{TabBliman}. As in \cite{Bliman2019,Bliman2023}, the carrying capacity of the wild mosquito population is scaled to 1 hectare by an appropriate choice of $\sigma$. While the entomological parameters are approximately known from previous publications \cite{Bliman2019,Bliman2023}, the values of the diffusion coefficients $\alpha_F, \alpha_M, \alpha_{M_S}$ can be estimated using biological insights as follows. 

Multiple Mark-Release-Recapture (MRR) studies reveal that \textit{Aedes aegypti} mosquitoes typically disperse only short distances in urban environments (for more details, see \cite{Harrington2005} and the references therein). In field experiments, most recaptures occurred at or near the release site, with mosquitoes remaining within 20-50 meters of their emergence site when nearby hosts and oviposition sites were present. Moreover, only a minority of individuals were moving beyond 100 meters from release. Notably, the MRR studies mentioned above did not distinguish between male and female flying distances, and therefore, we can assume that $\alpha_M \approx \alpha_F \approx \alpha_{M_S}$ without loss of generality.  

Based on such evidence, one can employ the formula that links the random flying distance after $t$ days, $R(t)$, to the constant diffusion coefficient $\alpha$ in two dimensions \cite{Berg1993}:
\[ R(t)  \approx \sqrt{4 \ \alpha \ t} \]
Thus, setting $\alpha=0.01$ hectare/day yields 20 meters with $t=1$ (flight distance after 1 day), around 63 meters with $t=10$ (after 10 days), and around 100 meters with $t=25$ days (during the mean lifespan of mosquitoes).   The latter justifies the values for $\alpha_M, \alpha_F,$ and $\alpha_{M_S}$ used in all numerical experiments and provided in Table \ref{TabBliman}. 

\begin{table}[t]
\centering
\begin{tabular}{ll}
  \hline \\  [-0.5ex]
  \textbf{Description} & \textbf{Assumed value}   \\ [1.5ex]
  \hline \\ [-1.5ex]
Fecundity rate of wild adult female mosquitoes & $\rho = 4.55$    \\ [0.5ex]
Primary sex ratio in offspring & $r=1-r=0.5$      \\ [0.5ex]
Parameter related to larval competition {and carrying capacity} & $\sigma = \frac{0.05}{140} = \frac{1}{2800}$  \\ [0.5ex]
Mortality rate of wild adult male mosquitoes & $\mu_M=0.04$         \\ [0.5ex]
Mortality rate of wild adult female mosquitoes & $\mu_F=0.03$ \\ [0.5ex]
Mortality rate of sterile adult male mosquitoes & $\mu_S=0.04$     \\ [0.5ex]
Relative mating efficiency of sterile male mosquitoes & $\gamma = 1$   \\ [0.5ex]
Diffusion coefficient for wild male mosquitoes & $\alpha_M = 0.01$    \\ [0.5ex]
Diffusion coefficient for wild female mosquitoes & $\alpha_F = 0.01$   \\ [0.5ex]
Diffusion coefficient for sterile mosquitoes & $\alpha_{M_S} = 0.01$   \\ [0.5ex]
  \hline
\end{tabular}
\caption{{\textit{Aedes aegypti}} parameters values borrowed from \cite{Bliman2019, Bliman2023}} \label{TabBliman}
\end{table}

For the following numerical simulations, we use the FEniCS finite element library~\cite{FEniCS}. We consider an evolution with final time $T=500$ days on the unit square $\Omega=()\times(0,1)$ (scaled to 1 hectare). Note that the theoretical results shown in the previous sections can be generalized to the case of such convex polygonal domains. For spatial discretization, we divide the square into a uniform triangular grid of size $h=\nicefrac{1}{64}$ (unless specified otherwise). 
For temporal discretization, we use a uniform step-size $\delta t = \nicefrac{1}{2}$ and choose $\theta=1$ (unless specified otherwise).

\subsection{Homogeneous Releases of Sterile Insects}
\label{subsec-homo}

Due to the homogeneous Neumann boundary conditions, the problem allows for spatially constant solutions, if $\Lambda, M_0, F_0$ and $M_{S_0}$ are constant. In this case, the solutions should correspond to solutions of the ODE model~\eqref{ODESys}. We will use this case to validate our first test computations. In~\cite[Theorem 3]{Bliman2019} the authors found that there exists a {threshold quantity} $\Lambda^{crit}$, {expressing the critical size of constant daily releases of sterile males,} such that for $\Lambda<\Lambda^{crit}$ there exist two different equilibria: one, where $M=F=0$, and another one, where both $M$ and $F$ remain positive. For $\Lambda>\Lambda^{crit}$, on the other hand, there exists only the mosquito-free equilibrium $M=F=0$. For the parameters in Table \ref{TabBliman}, $\Lambda^{crit}$ can be computed {numerically from Equation (11) in~\cite{Bliman2019}} as $\Lambda^{crit}\approx 1\:292$ by using Newton's method.

In Figure~\ref{fig:equ}, we show results for $\Lambda \in \{0.9\Lambda^{crit}, 1.1\Lambda^{crit}\}$ and $M_0=F_0\in\{80,85\}$. $M_{S_0}$ is chosen to be zero in all cases. We note that the initial sizes of the wild mosquitoes are deliberately assumed to be relatively small compared to their equilibrium values ($M^{*}=5\:194, F^{*}=6\:925$ with $\Lambda=0$, see \cite{Bliman2019}) to showcase the bistability for $\Lambda < \Lambda^{crit}$, inherited by the PDE model \eqref{Sys} from the ODE model \eqref{ODESys}. 

\begin{figure}[t]

\begin{minipage}{0.45\textwidth}\centering
\includegraphics[width=\textwidth]{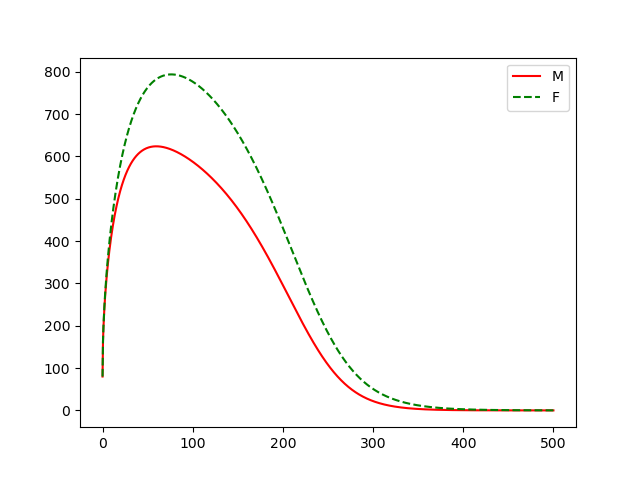}\\[-5.5cm]
{$\Lambda=0.9 \Lambda^{crit}, M_0=F_0=80$}
\vspace{5cm}
\end{minipage}\hfill 
\begin{minipage}{0.45\textwidth}\centering
\includegraphics[width=\textwidth]{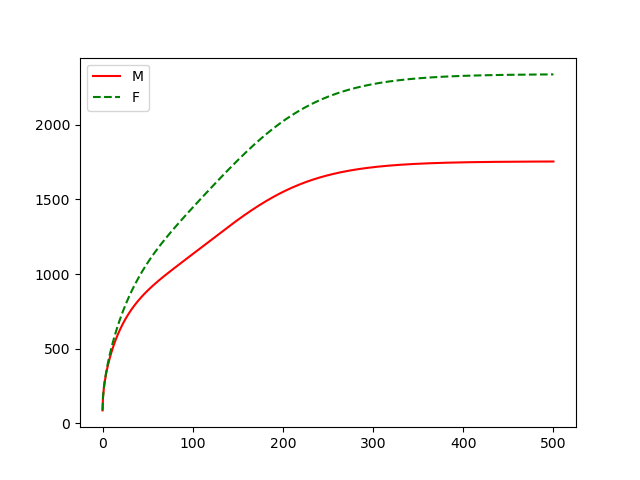}
\\[-5.5cm]
{$\Lambda=0.9 \Lambda^{crit}, M_0=F_0=85$}
\vspace{5cm}
\end{minipage}
\begin{minipage}{0.45\textwidth}\centering
\includegraphics[width=\textwidth]{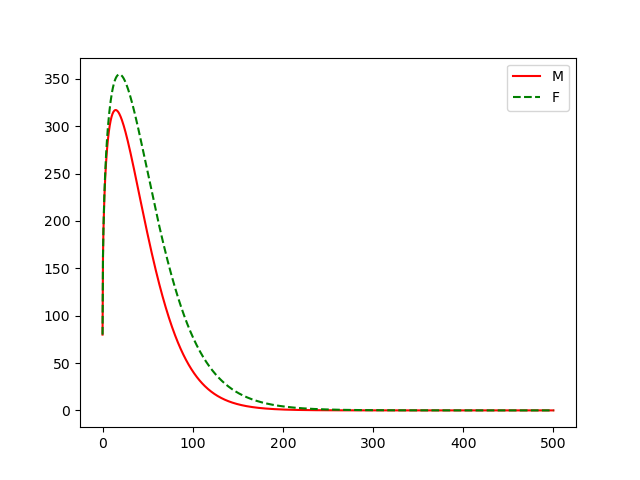}
\\[-5.5cm]
{$\Lambda=1.1 \Lambda^{crit}, M_0=F_0=80$}
\vspace{5cm}
\end{minipage}\hfill 
\begin{minipage}{0.45\textwidth}\centering
\includegraphics[width=\textwidth]{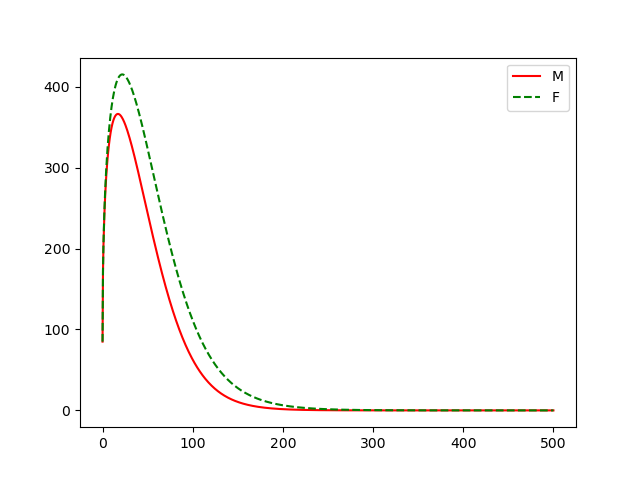}\\[-5.5cm]
{$\Lambda=1.1 \Lambda^{crit}, M_0=F_0=85$}
\vspace{5cm}
\end{minipage}
\caption{Evolution of $\big\|M \big\|_\Omega$ and $\big\|F \big\|_\Omega$ over time for constant releases $\Lambda =0.9\Lambda^{crit}$ (top) resp. $\Lambda =1.1\Lambda^{crit}$ (bottom) and $M_0=F_0=80$ (left) resp. $M_0=F_0=85$ (right). \label{fig:equ}}
\end{figure}

We observe that in all cases, {wild} mosquito populations grow fast in the beginning, when the population suppression induced by sterile males is not yet evident, because there is a notable delay between matings and the emergence (or non-emergence) of new adult individuals. The upper charts in Figure~\ref{fig:equ} illustrate the bistability for $\Lambda=0.9 \Lambda^{crit} < \Lambda^{crit}$. Namely, from smaller initial values ($M_0=F_0=80$, upper left chart), the system evolves toward a mosquito-free equilibrium, while from larger initial values ($M_0=F_0=85$, upper right chart), the system evolves toward a positive equilibrium whose coordinates are reduced compared to $M^{*}$ and $F^{*}$. In contrast, when $\Lambda=1.1 \Lambda^{crit} > \Lambda^{crit}$, the mosquito-free equilibrium is reached from both initial conditions (see two lower charts in Figure ~\ref{fig:equ}). These outcomes of the PDE system \eqref{Sys} align with the theoretical result ~\cite[Theorem 3]{Bliman2019} previously obtained for the ODE system \eqref{ODESys}.

Notably, the population of sterile mosquitoes $M_s$ is strictly increasing in all four cases displayed in Figure~\ref{fig:equ} due to their constant input, and reaches more than 29\:000 individuals at $t=500$ for $\Lambda =0.9\Lambda^{crit}$ and more than 35\:000 for $\Lambda =1.1\Lambda^{crit}$.  

\subsection{Nonhomogeneous Releases and Numerical Convergence Study}
\label{subsec.conv}

\begin{figure}[t]
\includegraphics[width=0.24\textwidth]{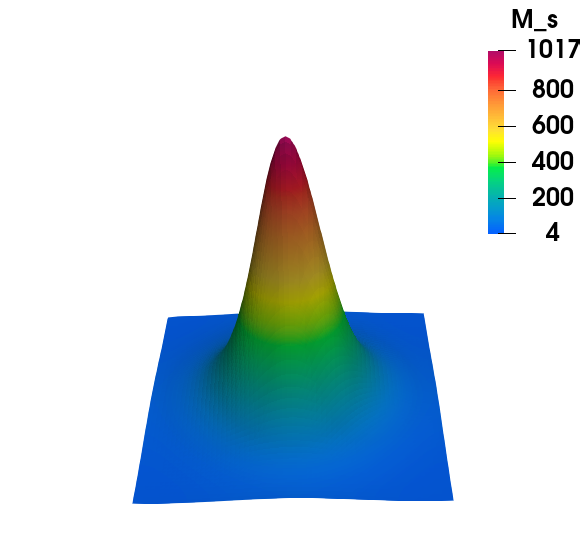}\hfil
\includegraphics[width=0.24\textwidth]{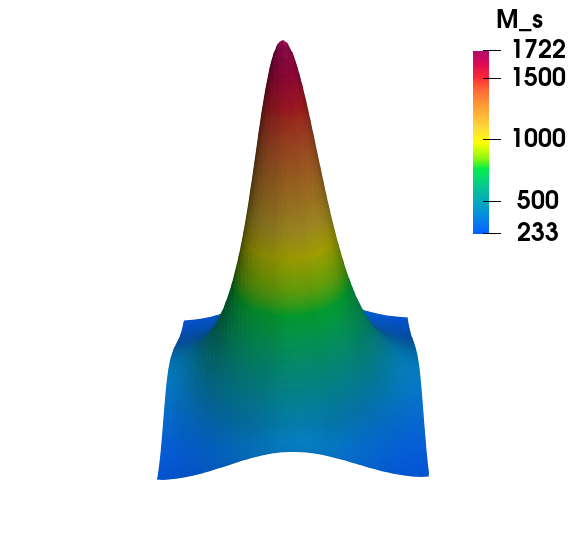}\hfil
\includegraphics[width=0.24\textwidth]{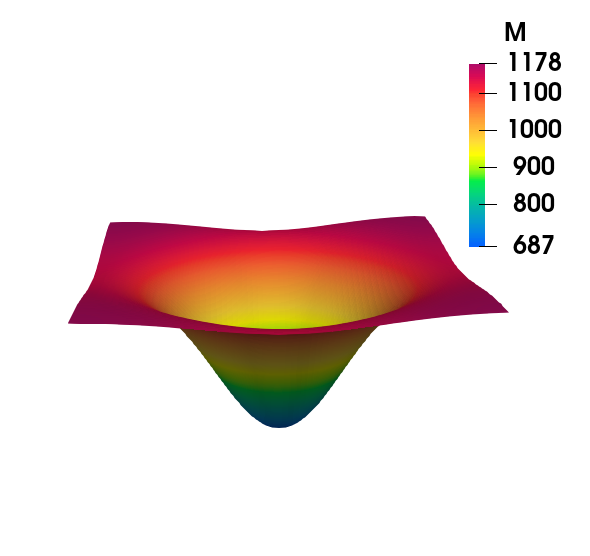}\hfil
\includegraphics[width=0.24\textwidth]{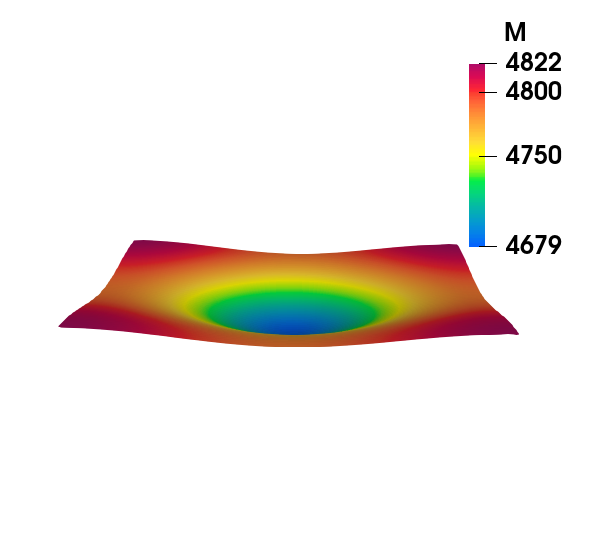}
\caption{\label{fig.res}\textit{Left}: Illustration of $M_S$ plotted over the domain $\Omega$ at times $t=2$ and $t=10$ (left to right), \textit{right}: Illustration of $M$ over $\Omega$ at times $t=2$ and $t=10$.}
\end{figure}
Next, we choose a space-dependent release. To this purpose, we define a Gaussian function centered in a point $p {=(p_1,p_2)}$ by
\begin{align}\label{defGp}
G_p(x,y) = \exp \big(-100 \big\|(x,y) - p \big\|_2^2 \big),  
\end{align}
and set with the center $m_p=(\nicefrac{1}{2}, \nicefrac{1}{2})$ of the domain $\Omega$
\begin{align}\label{LambdaG}
\Lambda(x,y) = 2000\, G_{m_p}(x,y).
    \end{align}
Notably, the $L^2$-norm of $\Lambda (x,y)$ defined by \eqref{LambdaG} is far less than $\Lambda^{crit}$. Therefore, the persistence of the wild population is expected, meaning that the time evolution of $M$ and $F$ should resemble the situation displayed in the upper right chart in Figure \ref{fig:equ}. Our intention is not only to showcase the impact of nonhomogeneous low-sized releases modeled by \eqref{LambdaG} on the spatial dynamics of the wild population, but also estimate the order of convergence of our algorithm for different values of $\delta t$ and $h$. 

The initial condition is fixed to  $(M_0,F_0,M_{S_0})=(100,100,0)$ to guarantee a high growth rate of the wild insect population at the beginning of the releases. Thus, we assume that the wild population is homogeneously distributed within the region $\Omega$ at $t=0$ and then wish to visualize its spatial distribution after sterile males are released nonhomogeneously.

Results of $M$  and $M_S$ at $t=2$ and $t=10$ are shown in Figure~\ref{fig.res}; the results for $F$ are similar to those of $M$.  In two days ($t=2$), the high density of sterile males in the center of $\Omega$ pushes out the newly emerged wild adults from the center of $\Omega$ towards its border, where the sterile insects are not yet present. However, in ten days ($t=10$), as more wild adults emerge from aquatic stages, the population of wild mosquitoes tends to recover its spatial homogeneity with only a slight reduction around the center of $\Omega$ due to the presence of sterile males. The results displayed in Figure~\ref{fig.res} bring forward two important outcomes captured by the PDE model \eqref{Sys}, namely: 
\begin{itemize}
\item 
Even small-sized, nonhomogeneous releases $\Lambda(x,y)$ induce spatial variability in the wild population, despite its initially homogeneous distribution.
\item 
Nonhomogeneous releases of the Gaussian form $\Lambda(x,y)$ with low overall sizes may only slightly reduce the equilibrium size of the wild mosquito population even if its initial density is relatively small. 
\end{itemize}
These insights will help us design more practical release strategies in the sequel (see Subsection~\ref{subsec-imp}).

To assess the order of convergence of the algorithm employed to obtain the results in Figure~\ref{fig.res}, we compare the $L^2$-norms of $M, F,$ and $M_S$ at $t=10$ for different values of the time step $\delta t$ and the spatial resolutions $h$.
In Table~\ref{Tabnx3} (left), we show the behavior of the $L^2(\Omega)$-norms of the three mosquito populations at $t=10$ days, where we fix the spatial resolution to $h=\frac{1}{64}$, but vary the time step $\delta t$.  The order of convergence $q$ is estimated by a least squares fit of the function $f(h)=f_0 + c \delta t^{q}$ for the three parameters $(f_0, c, q)$. As proven in Theorem~\ref{theo.conv}, we observe first-order convergence in all variables. In Table~\ref{Tabnx3} (right), we fix the time step to $\delta t=\frac{1}{80}$ and vary the spatial resolution $h$. We observe second-order convergence in all variables, which is optimal for $P_1$ finite elements.

\begin{table}[t]
\centering
\begin{tabular}{llll}
  \hline \\  [-0.5ex]
  % after \\: \hline or \cline{col1-col2} \cline{col3-col4} ...
  $\delta t$ & $\boldmath{\|M\|_{L^2}}$ & $\boldmath{\|F\|_{L^2}}$ & $\boldmath{\|M_{S}\|_{L^2}}$  \\ [1.5ex]
  \hline \\ [-1.5ex]
$\nicefrac{1}{10}$      &4784.36  &5062.73   &596.29 \\
$\nicefrac{1}{20}$     &4786.06  &5065.09  &596.72\\
$\nicefrac{1}{40}$    &4786.94  &5066.29   &596.94\\
$\nicefrac{1}{80}$  &4787.38  &5066.91  &597.05 \\ [0.5ex]
  \hline
  eoc & 0.97 &0.97 &0.98 \\ [0.5ex]
\end{tabular}\hfil
\begin{tabular}{llll}
  \hline \\  [-0.5ex]
  % after \\: \hline or \cline{col1-col2} \cline{col3-col4} ...
  $h$ & $\boldmath{\|M\|_{L^2}}$ & $\boldmath{\|F\|_{L^2}}$ & $\boldmath{\|M_{S}\|_{L^2}}$  \\ [1.5ex]
  \hline \\ [-1.5ex]
$\nicefrac{1}{16}$     &4786.66  &5066.04  &591.96\\
$\nicefrac{1}{32}$     &4787.22  &5066.72    &595.97\\
$\nicefrac{1}{64}$   &4787.38  &5066.91  &597.05\\
$\nicefrac{1}{128}$   &4787.42 &5066.95 &597.33\\ [0.5ex]
  \hline
  eoc & 1.84 &1.90 &1.90
\end{tabular}
\caption{$L^2$-norms of $M$, $F$, and $M_S$ at ${t}=10$ for different time step sizes $\delta t$ (left),  mesh sizes $h$ {(right),} and estimated order of convergence (eoc), assessed by a least-squares fit.} \label{Tabnx3}
\end{table}

\subsection{Time-stepping Schemes}

As we can see in Figure~\ref{fig.res} for $M$, the diffusion may have the effect that the solution tends towards a spatially constant solution for $t\to\infty$. It is known that the backward Euler scheme with $\theta=1$ (which is used for the linear terms in~\eqref{VFDP}) induces an additional numerical diffusion, {see Theorem~\ref{theo.conv} and} Remark~\ref{rem.diff}. In this section, we will investigate this effect by comparing results obtained for $\theta=1$ (where the implicit part is of backward-Euler type) and $\theta=0.5$ (Crank-Nicolson type), which is known to have perfect numerical dissipation properties. For further details, we refer to~\cite[Chapter 4.1.2]{RichterBuch}. 

We {now} set {more realistic but still spatially homogeneous initial conditions} $M_{S_0}=0$, $M_0=5\:000$ and $F_0=6\:700$, which {are} close to the equilibrium configuration in the absence of sterile mosquitoes, see~\cite{Bliman2019}. The release function $\Lambda$ is chosen as in the previous section, see~\eqref{LambdaG}, {to ensure persistence of the wild population.} We use two relatively large time-step sizes $\delta t =0.5$ and $\delta t=2$ (days) to make differences more visible. In Figure~\ref{fig.IEvsCN}, we plot the behaviour of $M$ over the diagonal line 
\begin{equation}
\label{G-line}
\Gamma_D := \Big\{(x,y)\in\Omega\, |\, x=y \Big\}, 
\end{equation}
{that runs} from the lower left to the upper right corner of the unit square {$\Omega$}, at times $t\in \{4,10,20,50\}$. At $t=4$, we observe significant differences between both schemes for the larger time-step size $\delta t=2$ (dashed lines). In particular, spatial differences in $M$ are much more pronounced for $\theta=0.5$. The {additional} numerical diffusion of the backward Euler method ($\theta=1$){, observed in Theorem~\ref{theo.conv},} leads to a much smoother curve. For later times $t$, however, the differences get smaller and smaller. Considering $\delta t=0.5$, we observe a similar, but much less pronounced effect at $t=4$. From $t\geq 10$ differences between the two schemes are barely visible. We conclude that the backward Euler-type scheme ($\theta=1$) can be used for the linear terms, if $\delta t\leq 0.5$.  

\begin{figure}[t]
\includegraphics[width=0.5\textwidth]{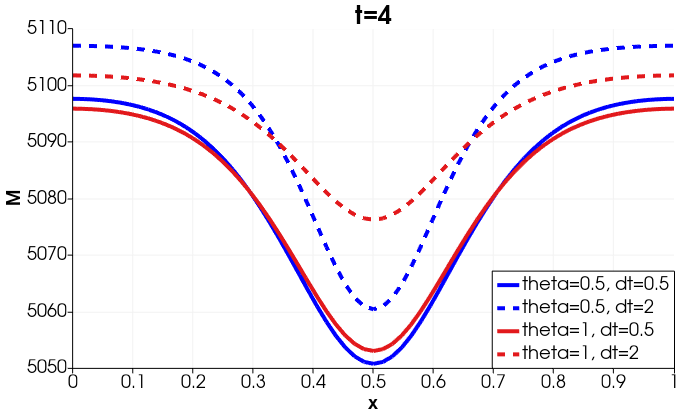}\hfil 
\includegraphics[width=0.5\textwidth]{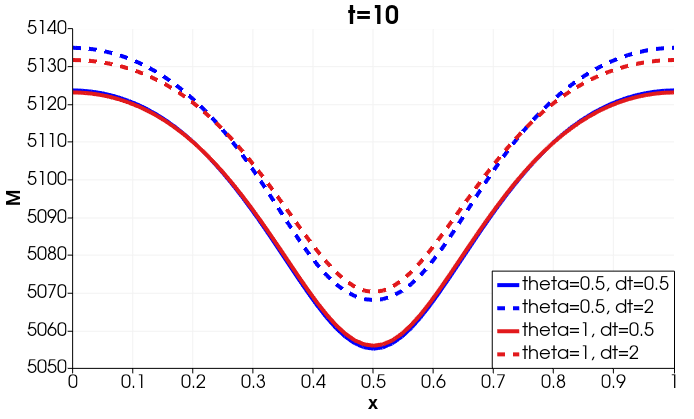}\\
\includegraphics[width=0.5\textwidth]{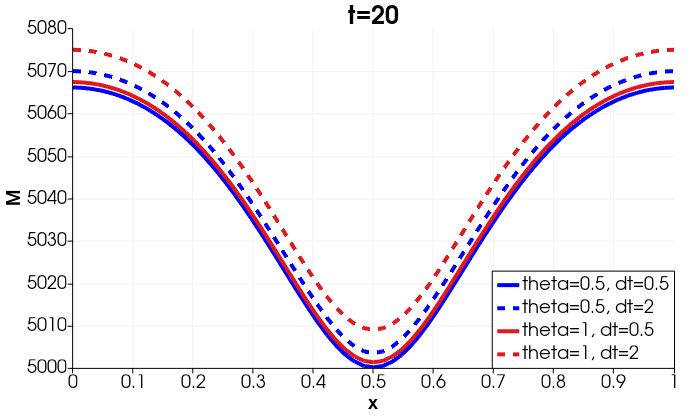}\hfil 
\includegraphics[width=0.5\textwidth]{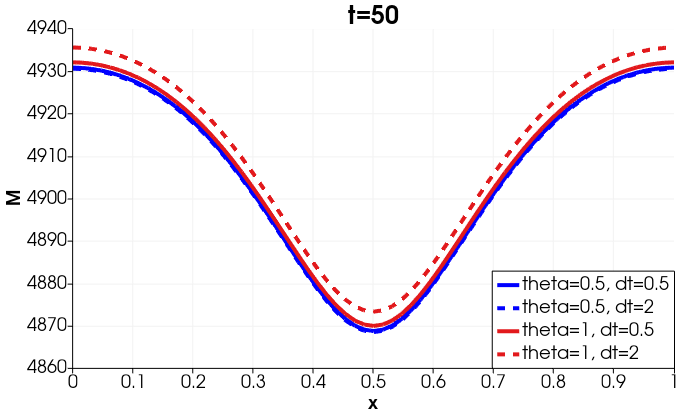}
\caption{\label{fig.IEvsCN}Plot over the diagonal line $\Gamma_D$ of $M$ at times $t\in \{4,10,20,50\}$ for different time-stepping schemes ($\theta$) and time-step sizes $\delta t$.} 
\end{figure}

In the following, we will therefore use $\theta=1$ and $\delta t =0.5$ in all simulations.

\begin{figure}[t]
\begin{tabular}{ccc}
{(a)} & {(b)} & {(c)} \\
\includegraphics[width=0.32\textwidth]{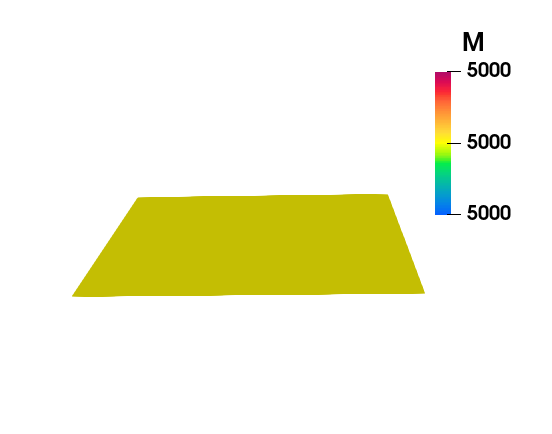}  %\hfil
& \includegraphics[width=0.32\textwidth]{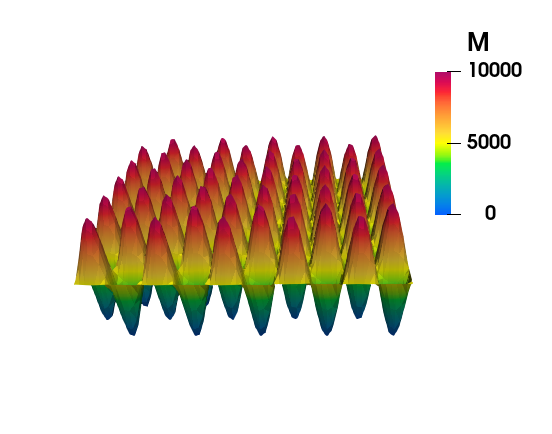} % \hfil
& \includegraphics[width=0.32\textwidth]{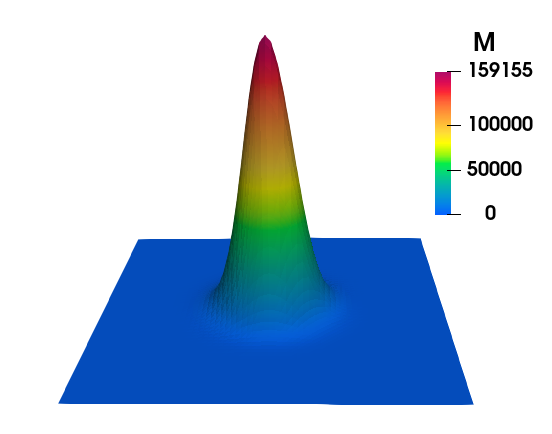}
\end{tabular}
\caption{\label{fig.initial}Visualization of the three initial conditions:  (a) uniform, (b) sinusoidal, (c) Gaussian.}
\end{figure}

\begin{figure}[t]
\centering
\includegraphics[width=0.5\textwidth]{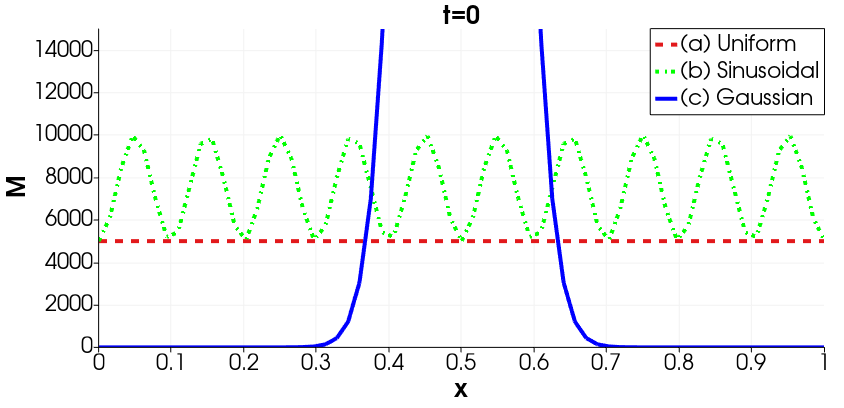}\hfil 
\includegraphics[width=0.5\textwidth]{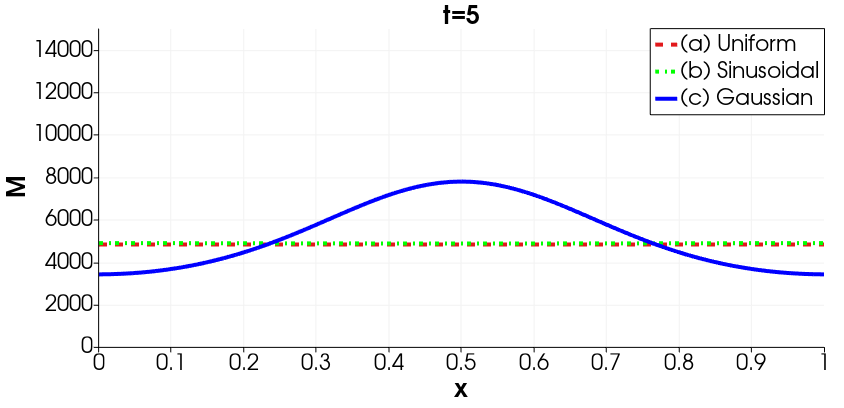}\hfil
\includegraphics[width=0.5\textwidth]{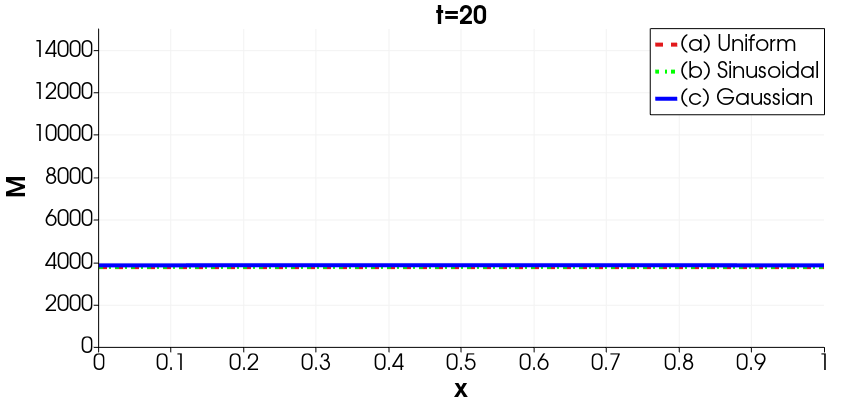}
\caption{\label{fig.initialstudy}Plot of $M$ over the diagonal line $\Gamma_D$ at times $t=0$, $t=5$ and $t=20$ (left to right) for different initial conditions $M_{0}$.}
    
\end{figure}

\subsection{Influence of Initial Conditions}
\label{sec-icon}

In Section \ref{subsec.conv}, it was revealed that spatially nonhomogeneous releases $\Lambda(x,y)$ may temporarily disturb the homogeneity of the initial conditions (see charts in Figure \ref{fig.res}). Still, they will eventually drive the wild populations to a spatially homogeneous equilibrium. In the next series of tests, we illustrate the inverse situation by visualizing whether spatially homogeneous (constant) releases $\Lambda < \Lambda^{crit}$ can bring an initially nonhomogeneous wild population to a positive equilibrium. 

We set $M_{S_0}=0$, $\Lambda=0.9\cdot \Lambda^{crit}$ and study the effect of different initial conditions for the wild mosquitoes $M_0(x,y)=5\,000\, g_{{i}}(x,y)$ and $F_0(x,y)=6\,700\, g_{{i}}(x,y), {i=1,2,3}$:

\begin{itemize}
\item[(a)] Uniform: $g_1(x,y)=1$
\item[(b)] Sinusoidal: $g_2(x,y)=1+\sin(10\pi x)\sin(10\pi y)$
\item[(c)] Gaussian: $g_3(x,y)=\frac{100}{\pi} G_{m_p(x,y)}$, where $m_p=(\nicefrac{1}{2}, \nicefrac{1}{2})$ and $G_p(\cdot)$ are defined in~\eqref{defGp} for $p\in\Omega$.
\end{itemize}

In all cases, the  functions {$g_i(x,y)$} are scaled, that is, 
\[ \int \limits_\Omega g_i(x,y)\, \text{d}x\text{d}y = 1, \quad  i =1,2,3. \] 
Thus, in all three cases, the initial densities of male and female mosquitoes are, on average, relatively close to their equilibrium values. Since $\Lambda$ is relatively large but still less than the critical threshold $\Lambda^{crit}$, we expect that the wild population, while being persistent, will eventually fall below its initial density as $t$ increases. 

The initial conditions are visualized in Figure~\ref{fig.initial}. Case (b) (sinusoidal) can be seen as a moderate perturbation of the constant initial data, while case (c) (Gaussian) represents a large perturbation with a high peak in the center. Intuitively, and based on the numerical tests performed in Section \ref{subsec-homo},  the PDE model \eqref{Sys} in the case (a) (constant spatially homogeneous initial conditions expressed by $g_1(x,y)=1$) is likely to behave similarly to the ODE model \eqref{ODESys} under constant spatially homogeneous releases $\Lambda < \Lambda^{crit}$. However, the wild populations $M$ and $F$ are expected to decrease in time, contrary to the plot displayed in the upper right chart of Figure \ref{fig:equ}, because the initial values of $M$ and $F$ are close to their natural equilibrium. Thus, case (a) will serve as the baseline for comparing the outcomes of the PDE model \eqref{Sys} initiated by the spatially nonhomogeneous initial conditions presented in charts (b) and (c) of Figure~\ref{fig.initial}. In the sequel, we present the results only for $M$, since those for $F$ are very similar.

First, we plot the initial density of male mosquitoes $M(x,y)$ over the diagonal line $\Gamma_D$, defined by \eqref{G-line}, for all three cases (a), (b), and (c) given in Figure~\ref{fig.initial}. Actually, {the plot over} $\Gamma_D$ can be visualized as an intersection curve between each of the surfaces displayed in Figure~\ref{fig.initial} and the vertical diagonal plane $x=y$. The resulting curves are displayed on the upper left chart of Figure~\ref{fig.initialstudy} by red dotted, green dashed, and blue solid lines for the cases (a), (b), and (c), respectively.  Notably, a closer look at the green dashed curve on this chart (sinusoidal case (b)) reveals that along the line $\Gamma_D$, the males exhibit the highest density at initial time $t=0$ compared to the other two cases (a) and (c).

However, at time $t=5$ (see the upper right chart in Figure~\ref{fig.initialstudy}), the density of male mosquitoes along the line $\Gamma_D$ in case (b) becomes almost identical to that of case (a). From a biological standpoint, this outcome of the PDE model \eqref{Sys} shows the effect of spatial diffusion of wild mosquitoes under constant homogeneous releases of sterile competitors, as the rivalry for mating opportunities drives them to seek an even distribution and become well-mixed with the sterile insects. 

A similar tendency to well-mixing is also observed in the Gaussian case (c). In fact, comparing the two upper charts in Figure~\ref{fig.initialstudy}, the blue solid curve shows a noticeable flattening and peak reduction between $t=0$ (left chart) and $t=5$ (right chart). However, the overall density of males along the line $\Gamma_D$ does not exhibit a notable reduction between $t=0$ and $t=5$, i.e., during the first five days since the beginning of the constant uniform releases of sterile males. Nonetheless, this outcome of the PDE model \eqref{Sys} agrees with the short-term population dynamics of mosquitoes since there is a delay between the appearance of sterile males in the system and the reduction in the quantity of newly emerged adult insects. The lower chart in Figure~\ref{fig.initialstudy} makes the role of sterile males more visible at $t=20$, where the overall population of wild males $M$ along the line $\Gamma_D$  not only becomes evenly distributed but also lower compared to its initial value in all three cases (a), (b), and (c).

From the above tests, we can conclude that the PDE model \eqref{Sys} adequately responds to uniform, constant releases of sterile males under any initial conditions, whether uneven or uniform. Even if the wild population before the releases is spatially nonhomogeneous, the uniform, constant releases of sterile males will eventually make the wild population homogeneously mixed within the area $\Omega$. The presence of spatial diffusion in the PDE model facilitates the mimicking of wild mosquito movements from spots with higher population density to spots with lower population density, thereby achieving a homogeneous distribution in response to a constant homogeneous inflow of sterile insects. 

From our numerical computations, we conclude that, due to spatial diffusion, the long-time behaviour of the PDE model \eqref{Sys} is not very sensitive to the initial data, at least if the control $\Lambda$ is constant and spatially uniform. In the next section, we will study the effect of space-dependent releases on wild mosquito populations with initially nonuniform spatial distribution.

\subsection{Location of the Control}
\label{subsec.loc}

\begin{figure}[t]
\centering
\begin{minipage}{0.48\textwidth}
\includegraphics[width=\textwidth]{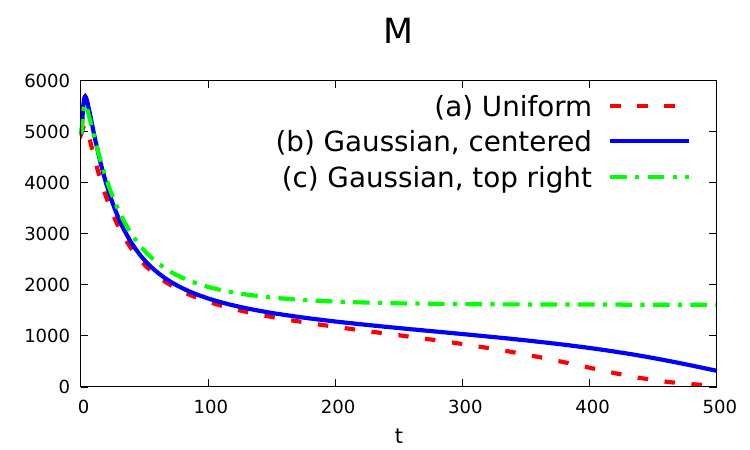}
\end{minipage}
\hfil 
\begin{minipage}{0.48\textwidth}
\includegraphics[width=\textwidth]{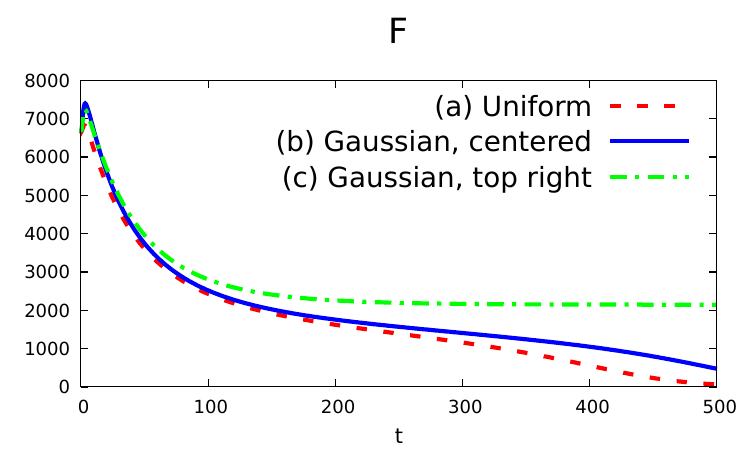}
\end{minipage}
\hfil 
\begin{minipage}{0.48\textwidth}
\includegraphics[width=\textwidth]{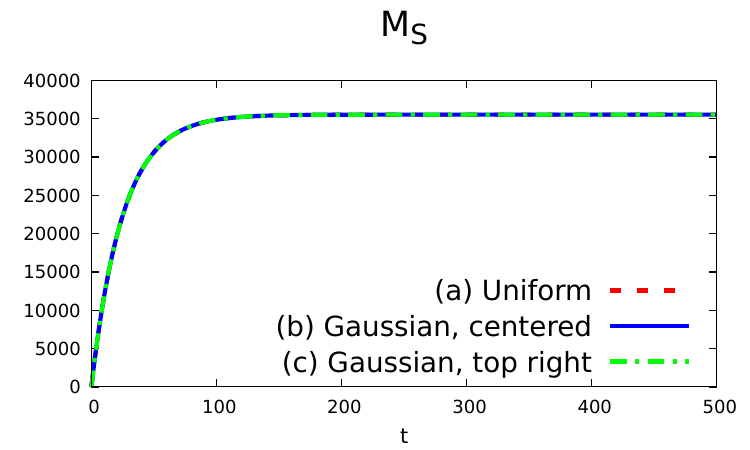}
\end{minipage}
\caption{\label{fig.cont_int} Behaviour of the total mosquito populations $\int \limits_\Omega M (x,y) \text{d}x\text{d}y$ (top left), \, $\int \limits_\Omega F(x,y) \text{d}x\text{d}y$ (top right), \, and $\int \limits_\Omega M_S(x,y) \text{d}x\text{d}y$ (bottom) over time $t$ (in days) for continuous releases. }    
\end{figure}

Next, we study the influence of the location of sterile mosquitoes. The initial data is chosen by means of the Gaussians
\begin{equation} \label{M0F0}
    M_0(x,y)=5\:000 \cdot \frac{100}{\pi} G_{m_p}(x,y), \qquad F_0=6\:700 \cdot \frac{100}{\pi} G_{m_p}(x,y), \quad \qquad M_{S_0}=0, 
\end{equation}
with $G_p(\cdot)$ defined in~\eqref{defGp} and $m_p=(\nicefrac{1}{2}, \nicefrac{1}{2})$.
This means that wild mosquitoes are initially concentrated around the midpoint of $\Omega$ with total densities close to equilibrium values.
We compare three controls $\Lambda{(x,y)}$, {which} are constant in time but differ in space:
\begin{itemize}
\item[(a)] Uniformly distributed: $\Lambda_a(x,y)=1.1\cdot\Lambda^{crit}$
\item[(b)] Gaussian around the center: $\Lambda_b(x,y)=\frac{100}{\pi} 1.1 \Lambda^{crit} G_{m_p}(x,y)$
\item[(c)] Gaussian centered on the top right: $\Lambda_c(x,y)=\frac{1}{0.0314031} 1.1 \Lambda^{crit} \cdot G_{p_{tr}}(x,y)$,
\end{itemize}
where $p_{tr}=(\nicefrac{3}{4}, \nicefrac{3}{4})$. All functions are scaled such that 
\[ \int \limits_\Omega \Lambda(x,y)\,\text{d}x\text{d}y = 1.1 \Lambda^{crit}, \] 
meaning that the daily release of sterile males exceeds the critical threshold value $\Lambda^{crit}$. Thus, with $\Lambda=\Lambda_a$ (spatially homogeneous releases), we can expect that the wild population will eventually be driven to extinction when $t$ becomes sufficiently large, see the results presented in Sections \ref{subsec-homo} and \ref{sec-icon}. To check this hypothesis, and also to visualize the impact of spatially nonhomogeneous releases $\Lambda_b (x,y)$ and $\Lambda_c (x,y)$ on the time evolution of populations $M, F$, and $M_S$, we present their long-term (500 days) plots in Figure~\ref{fig.cont_int}.

As displayed in the two upper charts of Figure~\ref{fig.cont_int}, the wild population will eventually become extinct under constant uniformly distributed releases $\Lambda_a(x,y) > \Lambda^{crit}$, which confirms our hypothesis. This, together with case {(a)} studied in Section~\ref{sec-icon} for constant homogeneously distributed releases $\Lambda < \Lambda^{crit}$, indicates that the PDE model \eqref{Sys} with constant uniform releases $\Lambda$ exhibits the same behavior as the ODE model \eqref{ODESys} even if the initial densities of the wild mosquitoes are non-uniform. 

The control $\Lambda_b(x,y)$, placed precisely over the wild population concentrated around the mid-point of $\Omega$ at time $t=0$, also ensures a steady decline of both $M$ and $F$ (see blue solid curves on the two upper charts in Figure~\ref{fig.cont_int}). However, $\Lambda_b(x,y)$ may apparently need a lot more time to eliminate the wild insects. 
Notably, by placing the control $\Lambda_c(x,y)$ farther from the mid-point of $\Omega$, around which the wild population is concentrated at $t=0$, we observe persistence of both $M$ and $F$ (see green dash-dotted curves on the two upper charts in Figure~\ref{fig.cont_int}) even though at considerably lower densities than initially.  

On the lower chart of Figure~\ref{fig.cont_int}, we observe that the total {size of the} population $M_S$ is equal {at each $t\geq 0$} in all cases, while their spatial distribution at each $t$ will depend on the control source functions (a)-(c). 

Thus, considering the time evolution of the wild mosquitoes $M$ and $F$ from their initial densities \eqref{M0F0} under three types of constant in size but spatially changeable release source functions (a)-(c), one may conclude from the results given in Figure~\ref{fig.cont_int} that elimination of the wild insects is
\begin{itemize} 
\item[-] 
achievable by equally distributed releases $\Lambda_a(x,y)$;
\item[-] 
possibly achievable in a longer time by the Gaussian-type centered releases $\Lambda_b(x,y)$, and
\item[-] 
unachievable under the biased Gaussian-type releases $\Lambda_c(x,y)$.
\end{itemize}

\begin{figure}[t]
\includegraphics[width=0.48\textwidth]{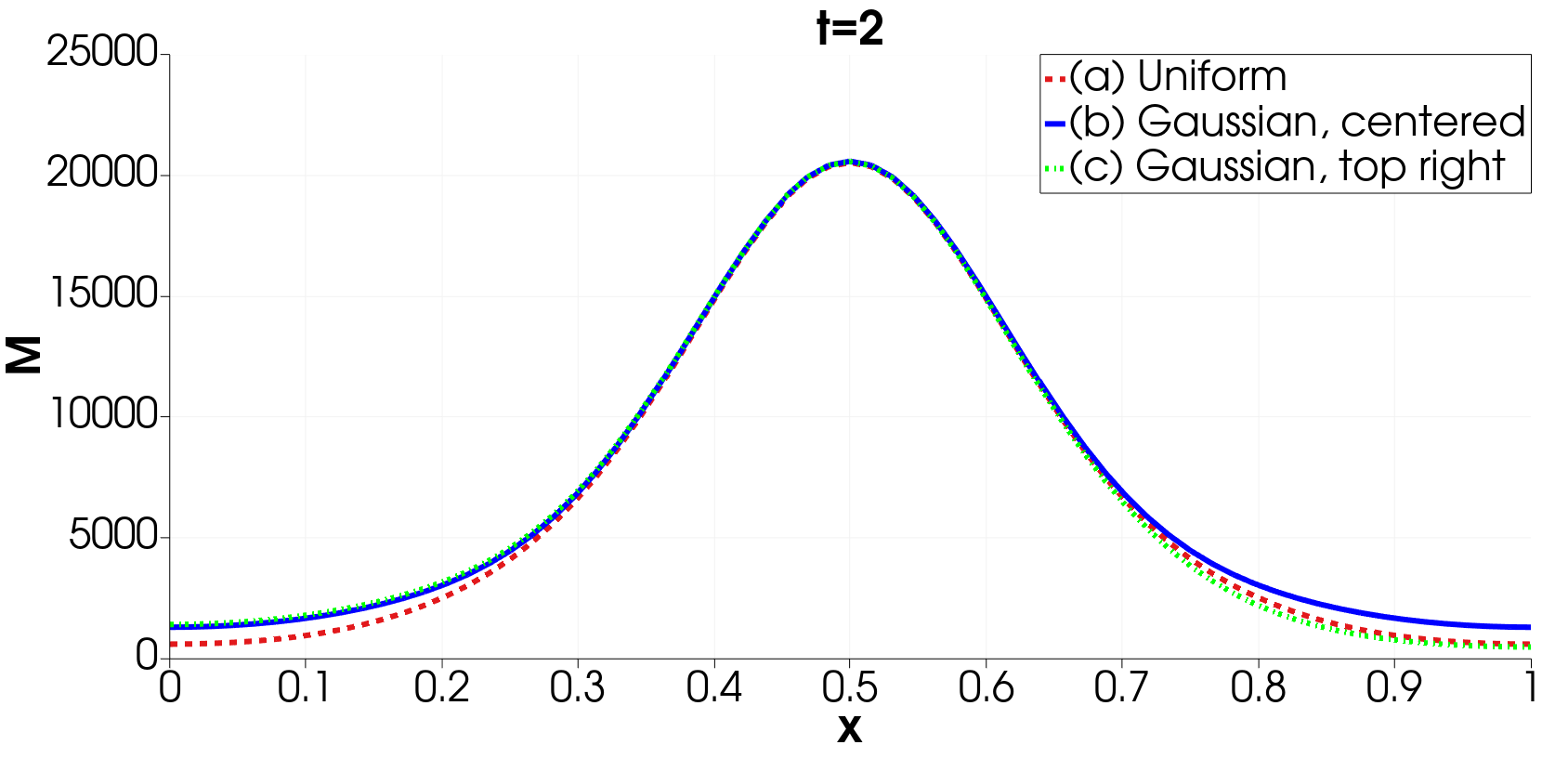}\hfil 
\includegraphics[width=0.48\textwidth]{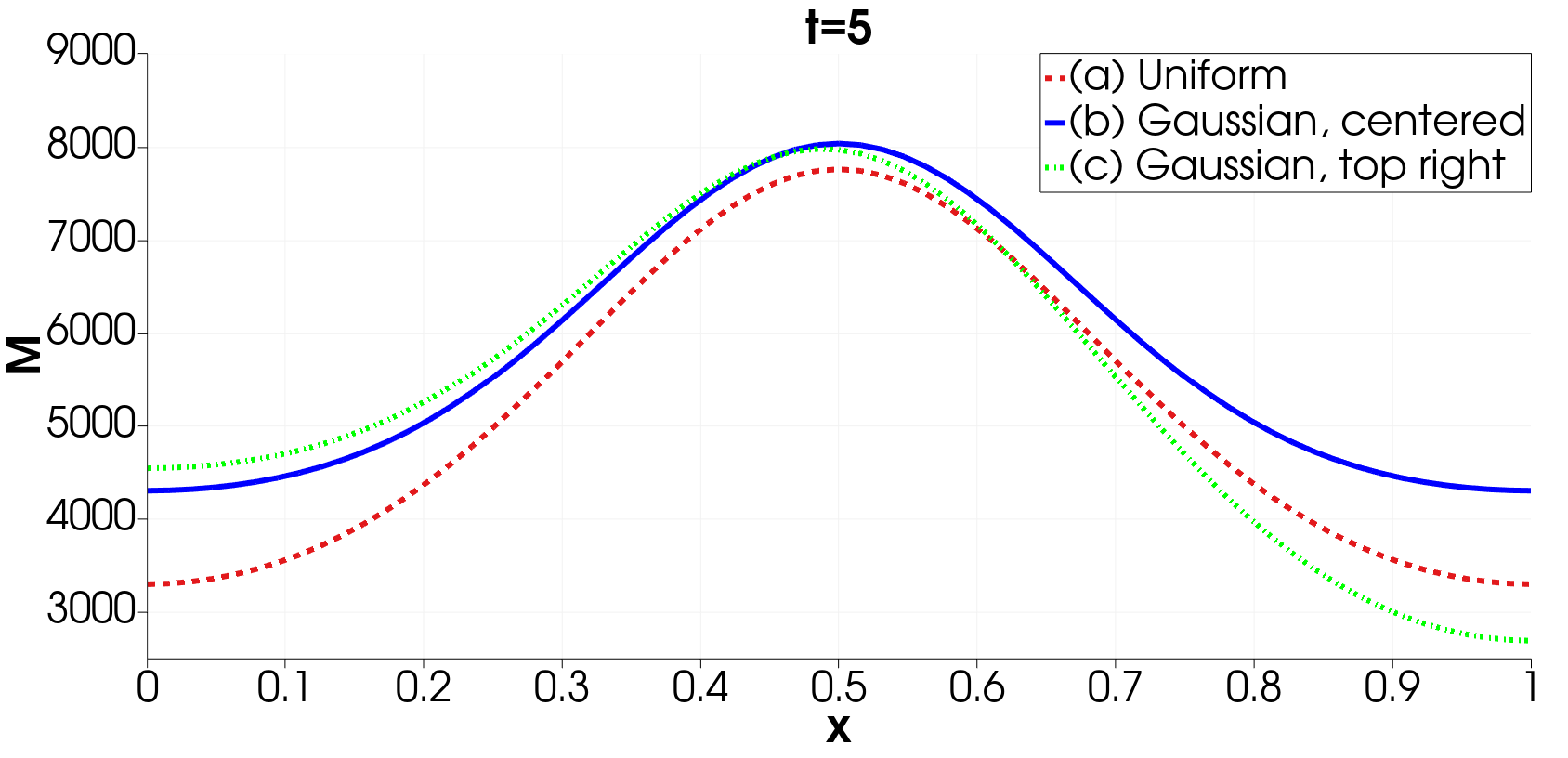}\\
\includegraphics[width=0.48\textwidth]{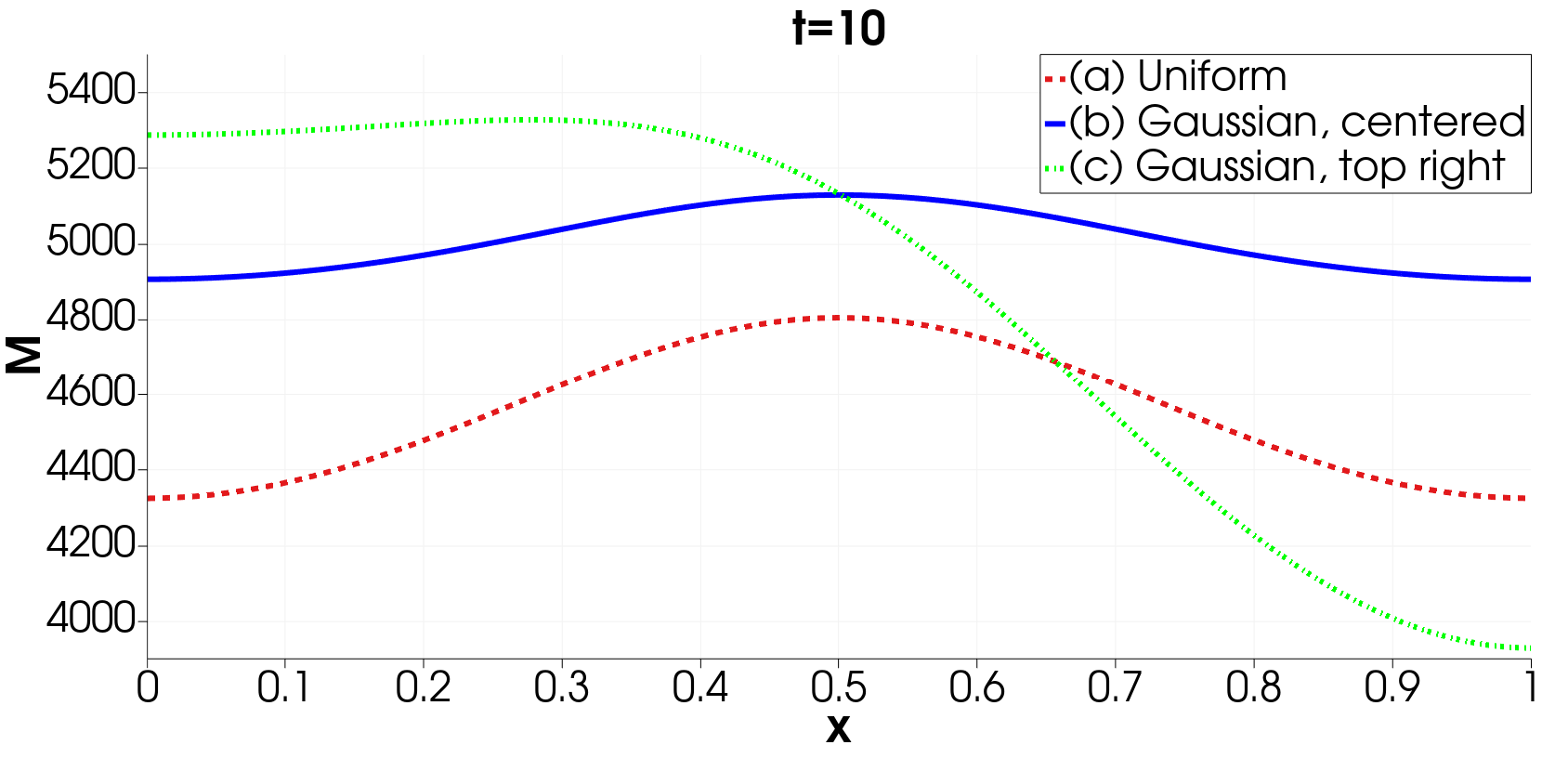}\hfil
\includegraphics[width=0.48\textwidth]
{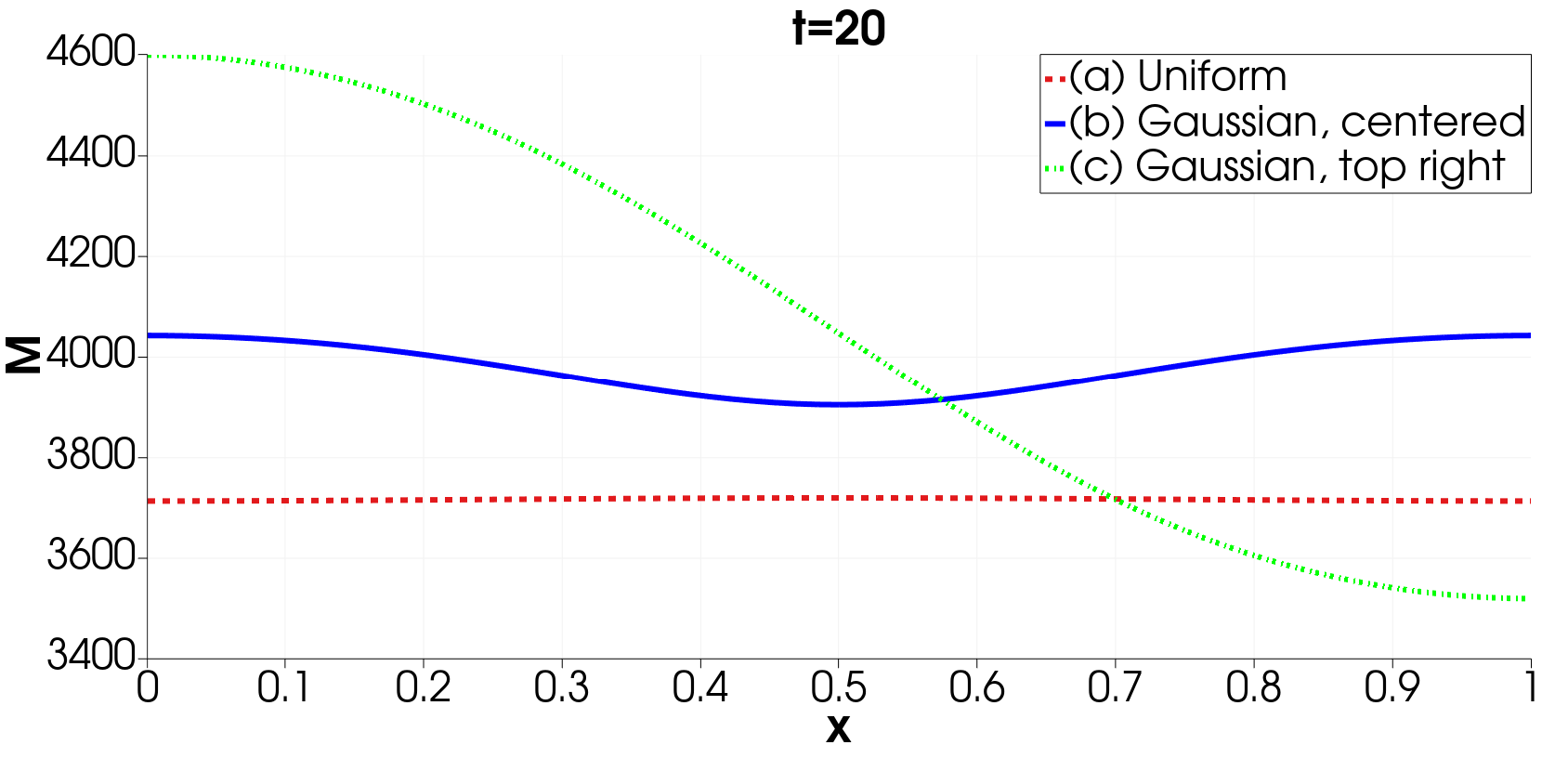}\\
\includegraphics[width=0.48\textwidth]{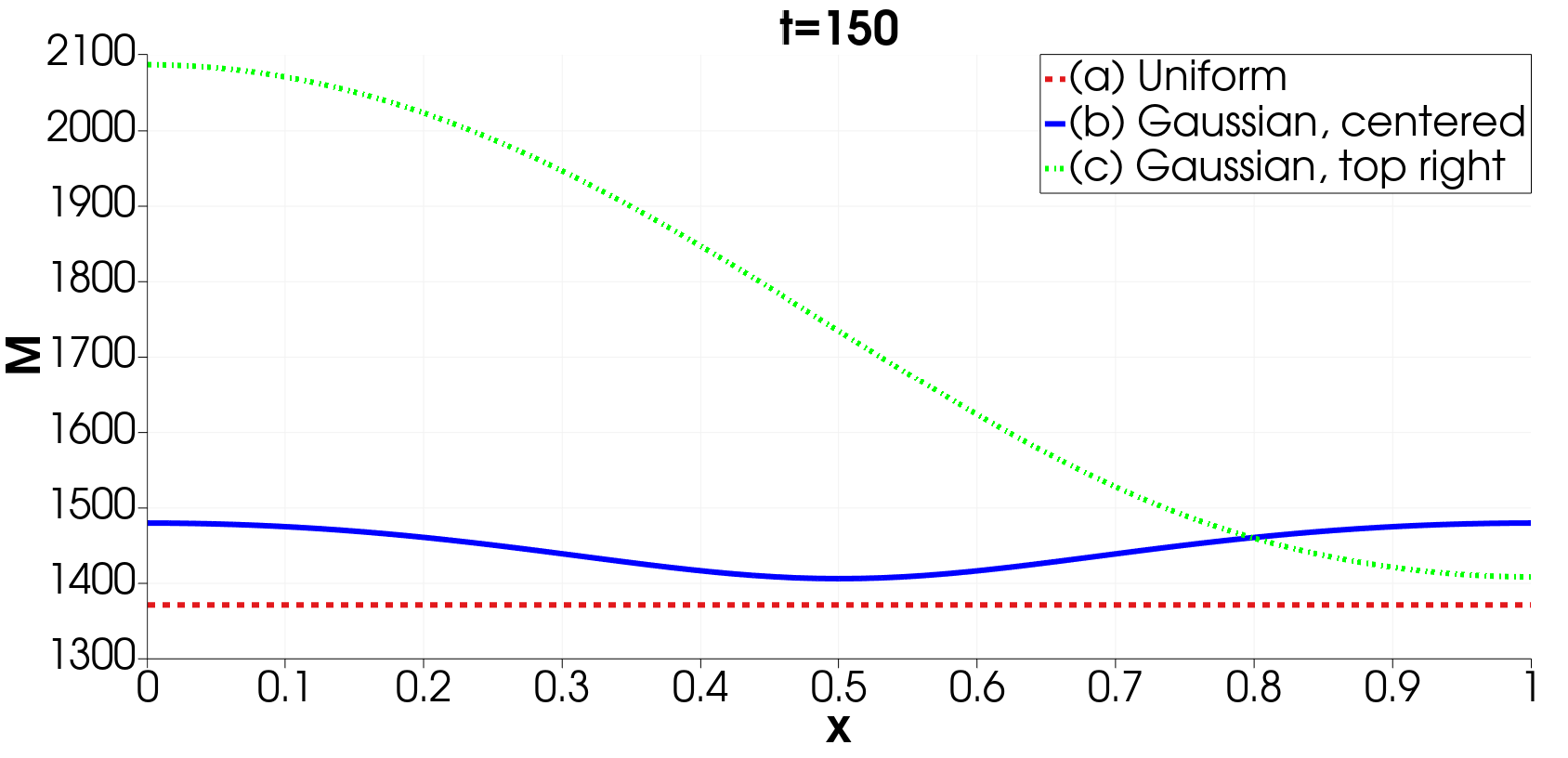}\hfil 
\includegraphics[width=0.48\textwidth]{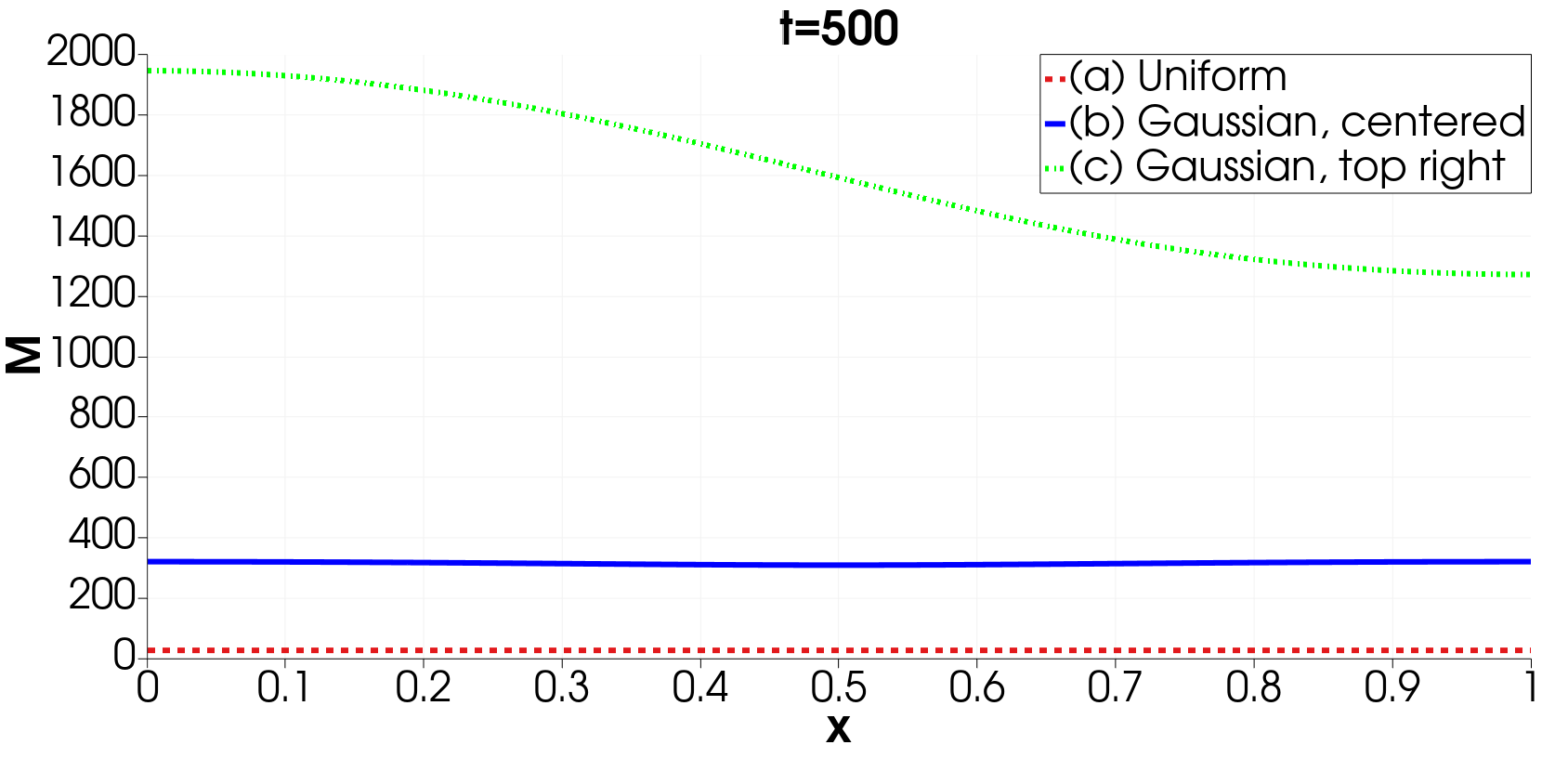} 
\caption{\label{fig.location} Plot of $M$ over the diagonal line $\Gamma_D$ for different spatial distributions of the (continuously released) control $\Lambda$ at times $t=2$ (top left), $t=5$ (top right), $t=10$ (mid left), $t=20$ (mid right), $t=150$ (bottom left) and $t=500$ (bottom right).}
\end{figure}

To gain further insight into the spatio-temporal dynamics, we plot in Figure~\ref{fig.location} the population of wild males $M$  over the diagonal line $\Gamma_D$, defined previously by \eqref{G-line}, under the three release types (a)-(c) at different time instants. 
We see that already at time $t=5$, the equally distributed control strategy (a) is more successful, even in reducing the initial peak of wild mosquitoes in the center, compared to strategies (b) and (c). Moreover, all charts in Figure~\ref{fig.location} attest that homogeneously distributed releases $\Lambda_a(x,y)$ are capable of maintaining the population of wild males $M$ at a lower level compared to nonhomogeneous releases $\Lambda_b(x,y)$ and $\Lambda_c(x,y)$ for all times, till its (almost) extinction at $t=500$ (cf. bottom right chart in Figure ~\ref{fig.location}). In this sense, the elimination of wild insects seems achievable through equally distributed releases $\Lambda_a(x, y)$ in 500 or more days. 

It is also helpful to contrast the right-hand middle chart in Figure~\ref{fig.location} ($t=20$) with the bottom chart in Figure~\ref{fig.initialstudy} from Section~\ref{sec-icon}. Both charts exhibit the distribution of the population $M$ along the diagonal line $\Gamma_D$ at $t=20$, starting from the same initial conditions \eqref{M0F0}, in response to constant spatially homogeneous daily releases $\Lambda$. However, in the former, the red dashed line depicts the distribution of the $M$-population under  $\Lambda > \Lambda^{crit}$, while in the latter, the blue solid line depicts the distribution of the $M$-population under  $\Lambda < \Lambda^{crit}$. As expected, a larger daily release size brings the population $M$ to a lower level (around 3\:700 individuals) than the smaller release size (around 4\:000 individuals), while in both cases and during the same time $t=20$ days, the population $M$ reaches spatially even distribution over the line $\Gamma_D$. This illustrates that wild mosquitoes disperse throughout the domain $\Omega$ and mix with sterile insects to balance their density in response to spatially uniform releases, regardless of the daily size of such releases.

When the control source $\Lambda_b(x, y)$ is placed precisely over the wild population, which is concentrated around the midpoint of $\Omega$, all insects disperse slowly over the region  $\Omega$ due to simple diffusion. However, as time increases, the areas located farther from the source of constant inflow of sterile males start producing a larger number of new wild adults compared to the area around the source (see blue solid curves in Figure~\ref{fig.location}). {In fact, at $t=20$ and $t=150$ the concentration of $M$ attains now a local minimum in the center, where initially the maximum of $M$ has been. This explains why the uniform strategy (a) is more successful in eliminating wild mosquitoes compared to case (b), where sterile mosquitoes are released at all times at the place where the local maximum has been in the beginning. Finally, in case (b)}
a rather low density {is reached} in a homogeneous distribution {at $t=500$}. This outcome agrees with the results of numerical tests given in Section~\ref{subsec.conv}, and further computation reveals that the $M$-population (as well as $F$-population) will slowly decline towards extinction. Thus, the elimination of wild insects can possibly be achieved by releasing sterile insects exactly over the concentration area of wild mosquitoes (that is, using $\Lambda_b(x, y)$), but in more than 500 days.

Finally, if the control source $\Lambda_c(x, y)$ is placed around the point $p_{tr}=(\nicefrac{3}{4},\nicefrac{3}{4})$ in $\Omega$, i.e., farther from the initial mass of the wild mosquitoes, all insects will gradually disperse over $\Omega$ due to simple diffusion, showing a pronensity to well-mixing. However, as time passes, the area farther from the influx of sterile males will produce considerably more new wild adults than the area around $p_{tr}$. Notably,  such a tendency will be maintained for a significant time (see green dash-dotted curves in all charts of Figure~\ref{fig.location}). 

A longer computation reveals that strategy (c) fails to {completely} eliminate the wild mosquitoes and converges to a {positive} equilibrium 
\[ \int \limits_\Omega M(x,y) \text{d}x\text{d}y \approx 1\:609 \quad \text{and} \quad \int \limits_\Omega F(x,y) \text{d}x\text{d}y \approx 2\:146. \]

Thus, we conclude that, based on the prediction of our PDE model \eqref{Sys}, choosing an appropriate release location {is} important in practice, as it determines whether wild mosquito elimination will eventually be achieved.

\subsection{Impulsive Control}
\label{subsec-imp}

In the preceding subsections, all our numerical tests simulated continuous-time releases of sterile males at a constant rate with or without spatial variability, which are commonly assumed in SIT-type models as a time-averaged representation of frequent release operations. This assumption facilitates analytical tractability and allows visualization of the model's results with respect to the explicit elimination threshold $\Lambda^{crit}$ derived for the ODE model \eqref{ODESys}. However, such releases are difficult to implement in practice because field deployments are inherently discrete, periodic in time (or irregular), and spatially constrained. Consequently, SIT-type models that feature only continuous releases should be interpreted as an idealized benchmark rather than an operational program, as they neglect the impulsive nature of field operations and may overestimate the effectiveness of SIT under realistic deployment constraints.

To bridge the gap between idealized continuous inputs and operational SIT protocols, we finally investigate the case in which large cohorts of newly emerged sterile male coevals are released all at once every 20 days. Therefore, we set $\Lambda=0$, but increase the sterile mosquito population every 20 days by 
\begin{align*}
M_S(t^+) = M_S(t^-) + \Lambda_{per}, \quad \text{if } \,t =20 k, \ k\in\mathbb{N}
\end{align*}
where $t^-$ ( resp. $t^+$ ) denotes the time at an infinitesimal instance before $t$ (resp. after $t$). Moreover, we set $M_{S_0}=\Lambda_{per}$, which means that the first release takes place at time $t=0$. The initial values for wild mosquitoes, $ M_0$ and $ F_0$,  are set as Gaussians centered at the midpoint
\begin{align}
M_0&=5\,000\cdot\frac{100}{\pi} G_{m_p}, \qquad
F_0=6\,700\cdot\frac{100}{\pi} G_{m_p},
\end{align}
which means that the initial densities of wild males and females are slightly below their equilibrium values.

We study three scenarios again, with different locations for the release of sterile mosquitoes:
\begin{itemize}
\item[(a)] {Uniformly distributed}: $\Lambda_{per}^a(x,y)=33\,000$
\item[(b)] Gaussian, centered: $\Lambda_{per}^b(x,y)=\frac{100}{\pi} 33\,000\, G_{m_p}$
\item[(c)] Gaussian, top right: $\Lambda_{per}^c(x,y)=\frac{1}{0.0314031} 33\,000\, G_{p_{tr}}$,
\end{itemize}
where the midpoint $m_p$ and $p_{tr}$ are defined as in the previous subsection.

From a practical perspective, implementation of Gaussian-type impulsive release strategies (b) and (c) is logistically easier and also cheaper compared to strategy (a) that requires reaching a sufficiently uniform spatial coverage throughout the target area $\Omega$. Nonetheless, recent advances in aerial release technologies have demonstrated that sterile male mosquitoes can be released over large areas with sufficiently high spatial homogeneity \cite{Bouyer2020}.

In Figure~\ref{fig.imp_int}, we show the course of the total populations of mosquitoes $\int_\Omega M {(x,y)} \text{d}x {\text{d}y}, \ \int_\Omega F {(x,y)} \text{d}x {\text{d}y},$ and $\int_\Omega M_S {(x,y)} \text{d}x {\text{d}y}$ over time. The total population of sterile mosquitoes is equal in the three cases and follows the same pattern, alternating between increases every 20 days (due to releases) and decreases (due to natural mortality). 
\begin{figure}[t]
\centering
\begin{minipage}{0.48\textwidth}
\includegraphics[width=\textwidth]{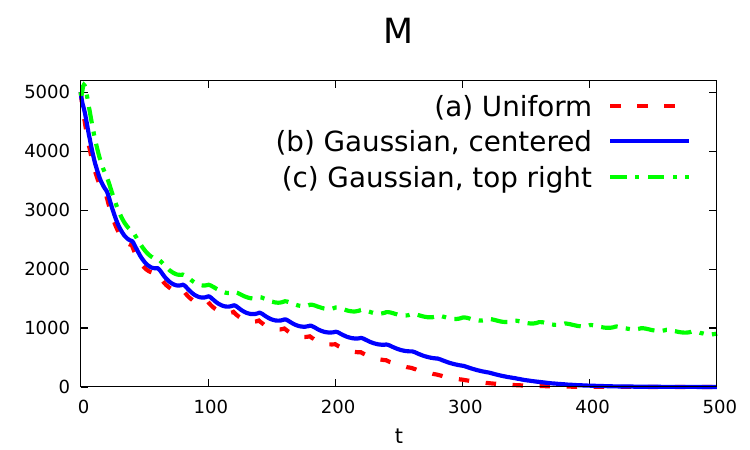}
\end{minipage}
\hfil 
\begin{minipage}{0.48\textwidth}
\includegraphics[width=\textwidth]{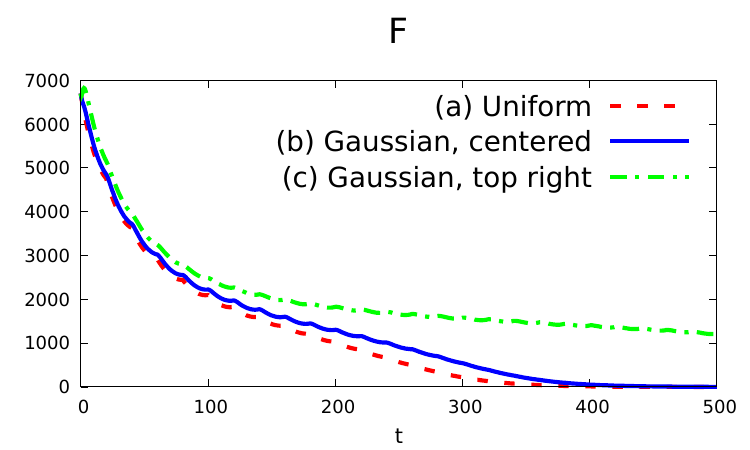}
\end{minipage}
\hfil 
\begin{minipage}{0.48\textwidth}
\includegraphics[width=\textwidth]{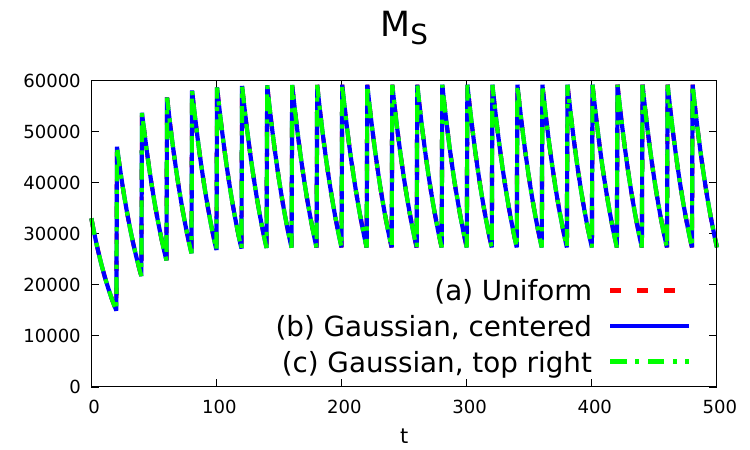}
\end{minipage}
\caption{\label{fig.imp_int} Behaviour of the total mosquito populations $\int \limits_\Omega M(x,y) \text{d}x {\text{d}y}$ (top left), \, $\int \limits_\Omega F(x,y) \text{d}x \text{d}y,$ (top right), \, and $\int \limits_\Omega M_S(x,y) \text{d}x \text{d}y,$ (bottom) over time $t$ (in days) for impulsive releases. 
}    
\end{figure}
The population of wild mosquitoes $M$ and $F$ decays almost monotonically for all strategies, showing a similar rate of decay during the first 50 days, which further slows down more for strategy (c) than for (b) and (a). Such a pronounced initial decline can be explained by overshooting the carrying capacity, due to a large simultaneous release of sterile insects.

A closer look at the curves for $M$ and $F$ under strategy (c) also reveals a small rise in the first three days (see the green dash-dotted lines in the two upper charts in Figure~\ref{fig.imp_int}). Indeed, the population of wild mosquitoes, initially concentrated around the midpoint $m_p$, continues to grow toward equilibrium. At the same time, sterile males, released relatively far from $m_p$, have not yet reached the center of $\Omega$ by natural diffusion.

From the two upper charts in Figure ~\ref{fig.imp_int}, we deduce that wild mosquitoes can be completely eliminated by strategies (a) and (b), while strategy (c) fails to do so within the time frame  considered. Namely, after $T=500$ days, the remaining wild population stands at $F(T)\approx 1\,220$ and $M(T)\approx 906$. This happens because sterile mosquitoes are released in the ``wrong'' location, that is, too far from the concentration of the wild population. However, a longer calculation reveals that the elimination will ultimately take place on a much longer timescale. In fact, the population of wild mosquitoes  drops below 1 from $t\geq 925$, if the releases continue every 20 days. 

In general, the equally distributed strategy (a) performs best from $t\geq 20$.
To explain this outcome, we should contrast the spatial distribution of the sterile and wild males shortly after a massive impulsive release and immediately before the next one under the three strategies (a), (b), and (c). 

\begin{figure}[t]
\begin{tabular}{ccc}
%{$t=0$} & {$t=5$} & {$t=20^{-}$} \\
\includegraphics[width=0.33\textwidth]{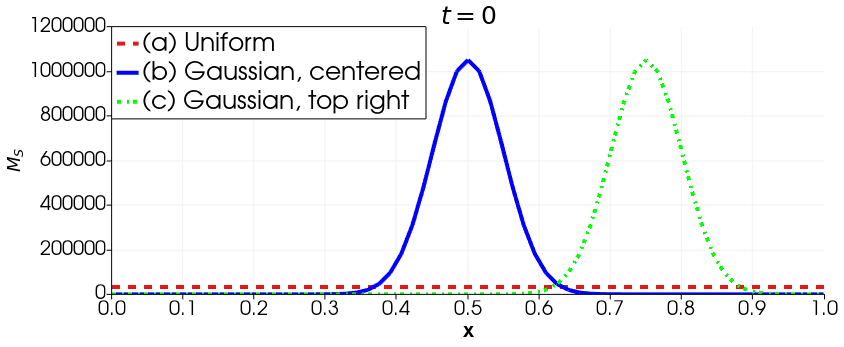} %\hfil
& \includegraphics[width=0.33\textwidth]{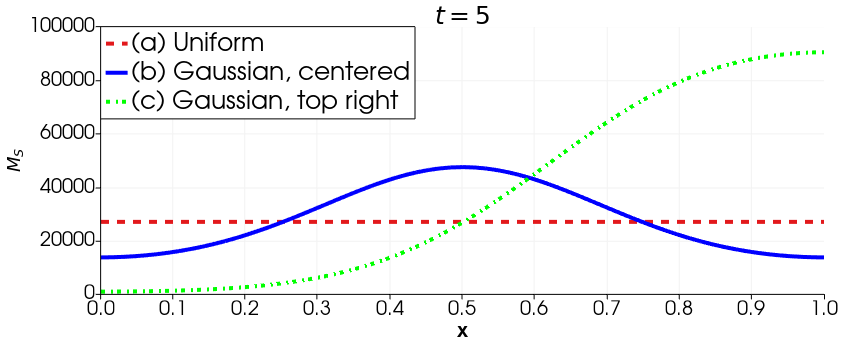} %\hfil
& \includegraphics[width=0.33\textwidth]{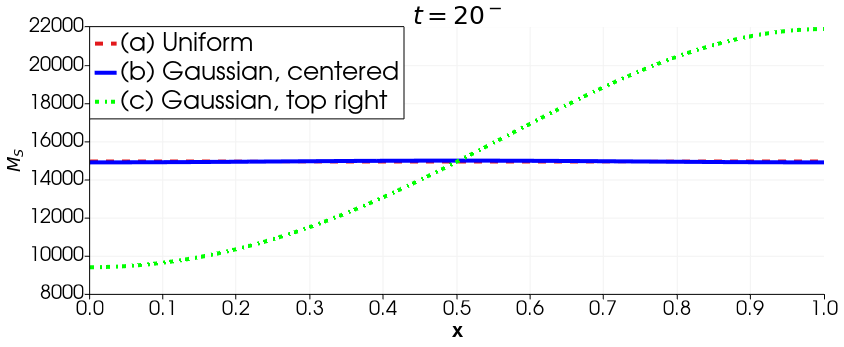} \\ %\hfil
%{$t=25$} & {$t=40^{-}$} & {$t=45$} \\
\includegraphics[width=0.33\textwidth]{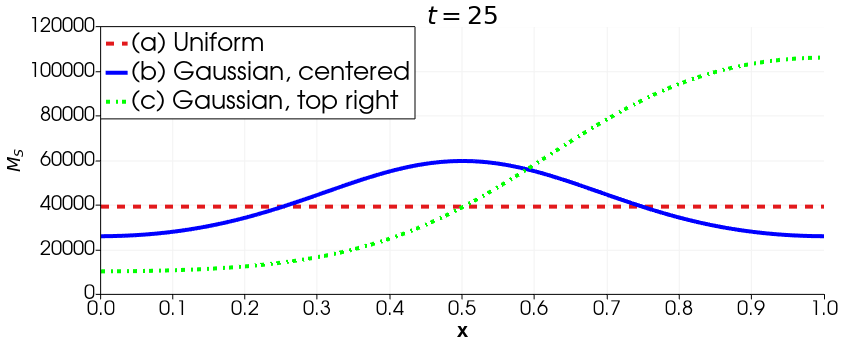} %\hfil
& \includegraphics[width=0.33\textwidth]{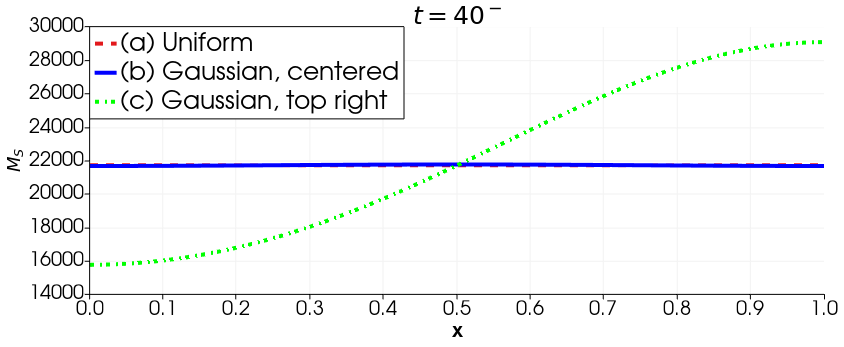} %\hfil
& \includegraphics[width=0.33\textwidth]{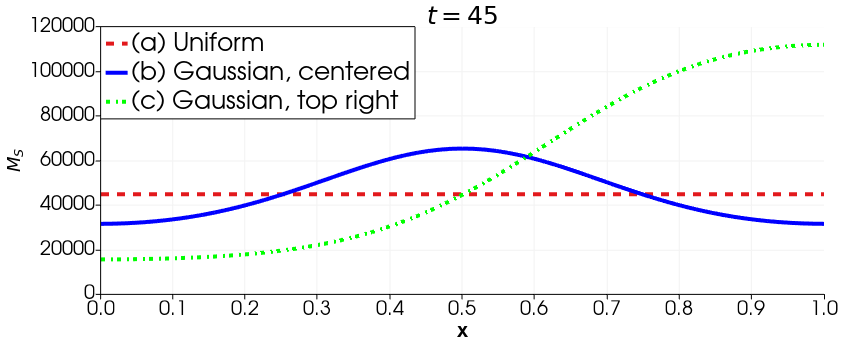}
\end{tabular}
\caption{\label{fig.imp_space_MS}Plots of $M_S$ over the diagonal line $\Gamma_D$ at times $t=0$, $t=5$, $t=20^-$, $t=25$, $t=40^-$, and $t=45$ (top left to bottom right) for the three different impulsive control strategies.}
\end{figure}

\begin{figure}[t!]
\begin{tabular}{ccc}
%{$t=0$} & {$t=5$} & {$t=20^{-}$} \\
\includegraphics[width=0.33\textwidth]{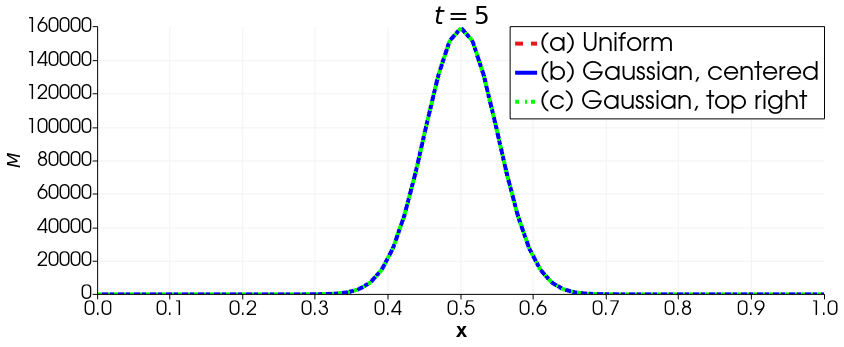} %\hfil
& \includegraphics[width=0.33\textwidth]{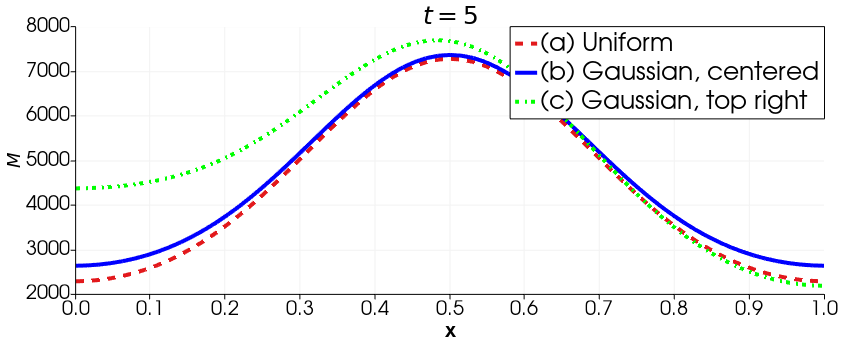} % \hfil
&  \includegraphics[width=0.33\textwidth]{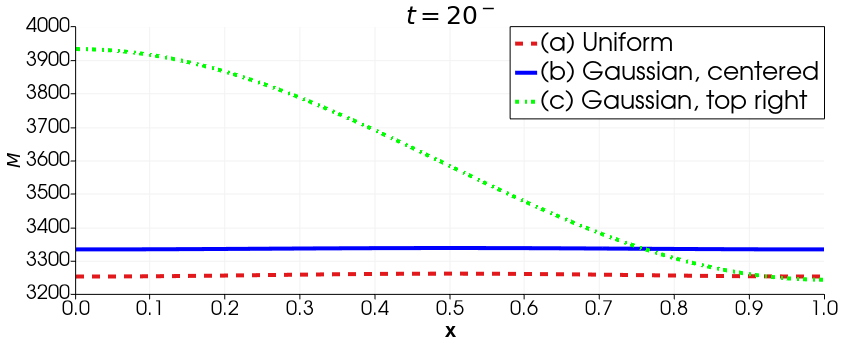} \\ %\hfil
%{$t=25$} & {$t=40^{-}$} & {$t=45$} \\
\includegraphics[width=0.33\textwidth]{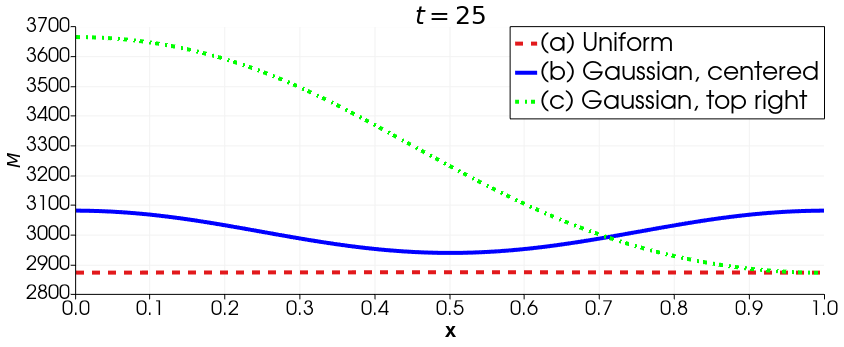} %\hfil
& \includegraphics[width=0.33\textwidth]{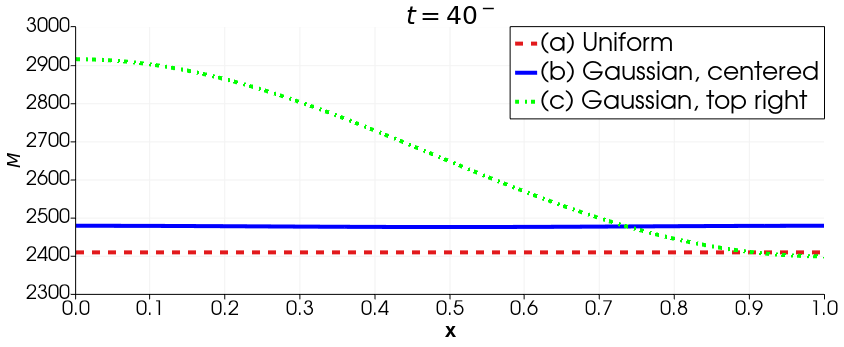} %\hfil
&\includegraphics[width=0.33\textwidth]{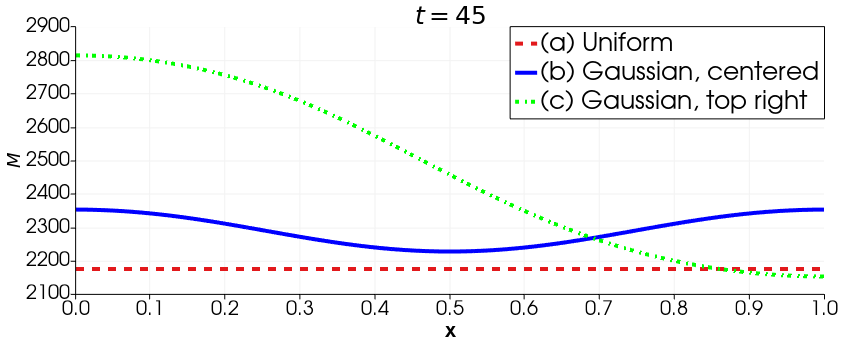}
\end{tabular}
\caption{\label{fig.imp_space} Plots of $M$ over the diagonal line $\Gamma_D$ at times $t=0$, $t=5$, $t=20^-$, $t=25$, $t=40^-$, and $t=45$ (top left to bottom right) for the three different impulsive control strategies.}
\end{figure}

To this purpose, we turn to Figures~\ref{fig.imp_space_MS} and~\ref{fig.imp_space} depicting, respectively, the behavior of $M_S$ and $M$, across the diagonal line $\Gamma_D$ defined by \eqref{G-line} at times $t \in \{ 0, 5, 20^{-}, 25, 40^{-}, 45 \}$  for the three control strategies (a), (b), and (c). Correlating the charts in Figures~\ref{fig.imp_space_MS} and~\ref{fig.imp_space}, one can see how massive space-dependent impulsive releases of sterile males may affect the spatial distribution and density of the wild population. Here, it is worth mentioning that the population of wild females, $F$, behaves very similarly to the population of the wild males, $M$, and, therefore, the model's outcomes for $F$ are not displayed. 

Concerning $M_{{S}}$ (Figure~\ref{fig.imp_space_MS}), we observe that the initial peak in strategy (b) diffuses quickly across the entire domain, resulting in a uniform distribution of sterile males before the next release at $t=20^{-}$. In contrast, the initial peak in strategy (c) will spread around the upper right corner of $\Omega$ until the consequent release at $t=20^{-}$ (note also the different scalings on the vertical axis evincing that the total population of sterile males decreases over the time interval $[0, 20^{-}]$). 

The effect of the initial impulsive release of sterile males, performed by the three different strategies (a), (b), and (c), on the population of wild males is illustrated in the three upper charts of Figure \ref{fig.imp_space}. Here, we observe that the uniform and centered Gaussian release strategies (a) and (b) act similarly upon the wild population, leading to its even dispersal at comparable total densities just before the consequent release at $t=20^{-}$. The biased Gaussian-type initial impulsive release (strategy (c))   leads, however, to dispersal of the wild males towards the opposite side of the domain before the next release, and the total density of $M$ at $t=20^{-}$ is considerably higher compared to the cases (a) and (b). The latter indicates that a single massive release in the ``wrong place'' does not reduce the wild population. 

The effect of the two subsequent impulsive releases (at $t=20$ and $t=40$) on the population distribution of both $M_S$ and $M$ along the diagonal line $\Gamma_D$ is shown in the three lower charts of Figures~\ref{fig.imp_space_MS} and~\ref{fig.imp_space}. In Figure ~\ref{fig.imp_space_MS}, we observe that, under the centered Gaussian strategy (b), the peak becomes visible again at $t=25$, after the second release of sterile mosquitoes, and also at $t=45$ after the third one. Furthermore, at $t=40^-$, the sterile mosquitoes are spread again over the entire domain, very similarly to $t=20^{-}$, thus indicating convergence towards a periodic state. When sterile males are released in the upper right corner of $\Omega$ (strategy (c)), we observe a high density of $M_S$ in this corner before and after the subsequent releases. However, we also see that the density of sterile insects in the opposite corner of $\Omega$ (i.e., close to the origin) slowly increases as time passes, evincing the tendency to dispersal of the $M_S$-population across the entire region as well. 

To understand the impact of the second and third impulsive releases on the distribution of the wild population, we turn to the three lower charts in Figure ~\ref{fig.imp_space}. Even though just before the second release, at $t=20^-$, the {wild} population was almost homogeneous for all strategies, a few days after the second release (at $t=25$), the $M$-population exhibited a decrease on the right for strategy (c), in the center for strategy (b) and uniformly for strategy (a). Notably, a local minimum in the center can be observed for the control strategy (b) at $t=25$ and also at $t=45$ (five days after each impulsive release), indicating that the uniformly distributed strategy (a) is more efficient. It is also worth noting that after the second release, the total density of the wild population declines as time passes under all three strategies, with strategy (a) acting as the fastest and strategy (c) being the slowest. 

For better illustration, we show the complete spatial distributions of $M$ over $\Omega$ at $t=25$ in Figure~\ref{fig.imp_T5}. The ranges of variation can be inferred from the different color scales. While for strategy (a), the population is almost uniform, we observe a local minimum in the center for strategy (b) and a minimum in the top right corner for strategy (c), both corresponding to the location where sterile mosquitoes have been released. 
%
% to continue

At time $t=40^-$, the population {of wild males} is again almost {uniform} in cases (a) and (b) (due to diffusion), but still bears a smaller average value under strategy (a). For strategy (c), we observe a minimum on the right, where the sterile mosquitoes still keep their maximal density (cf. the bottom rows in  Figures~\ref{fig.imp_space_MS} and~\ref{fig.imp_space}). At $t=45$ (after the third release at $t=40$), the spatial distribution of the $M$-population is similar to that at $t=25$, but with a lower average density. Overall, the wild mosquitoes are reduced very gradually in the bottom left part of the region $\Omega$ under strategy (c), which leads to the slow elimination observed in Figure~\ref{fig.imp_int}, compared to the other two cases (a) and (b).

Our simulations thus revealed that impulsive releases (studied in this subsection) and constant releases (presented in Subsection \ref{subsec.loc}) induce spatial behavior that is rather similar in both wild and sterile insects. However, from a practical standpoint, implementing impulsive releases is much more feasible than constant releases.

\begin{figure}[t]

\begin{minipage}{0.33\textwidth}
\includegraphics[width=0.93\textwidth]{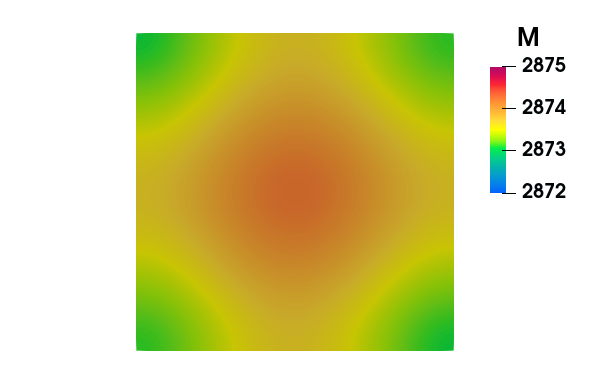}
\end{minipage}\hfil 
\begin{minipage}{0.33\textwidth}
\includegraphics[width=\textwidth]{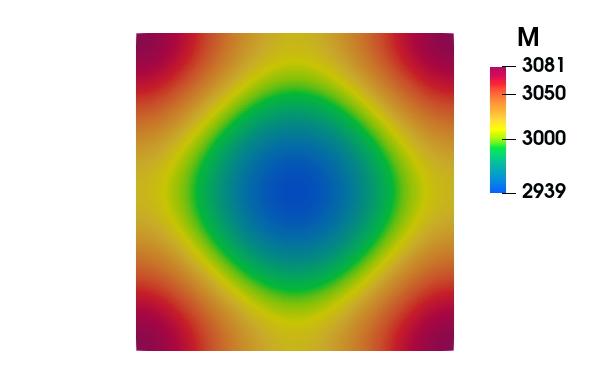}
\end{minipage}\hfil 
\begin{minipage}{0.33\textwidth}
\includegraphics[width=\textwidth]{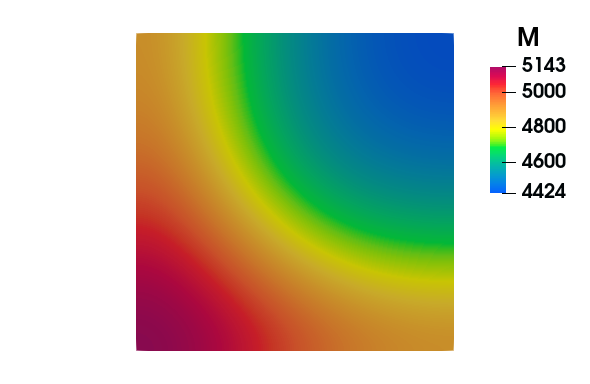}
\end{minipage}
\caption{\label{fig.imp_T5} Spatial distribution of $M$ at $t=25$ for strategies (a), (b) and (c) (left to right).}
\end{figure}

\section{Conclusions}
\label{sec-conc}

In the present work, we developed and analyzed a {spatiotemporal} reaction–diffusion model describing the dynamics of wild male, wild female, and sterile male \textit{Aedes aegypti} mosquitoes under Sterile Insect Technique (SIT) interventions. The model combines sex structure, Ricker-type recruitment, and two-dimensional diffusion into a biologically consistent and mathematically rigorous framework for spatial SIT analysis. Unlike logistic formulations that impose density dependence on adult mortality, our approach places competition at the recruitment stage, consistent with \textit{Aedes aegypti} biology and naturally leading to a threshold mechanism governed by mating competition and recruitment limitation rather than adult crowding effects.

From an analytical standpoint, we established global existence, uniqueness, nonnegativity, and boundedness of solutions, ensuring biological consistency and well-posed spatio-temporal dynamics under operationally relevant release strategies. On the numerical side, we proposed and analyzed a semi-implicit time discretization combined with a spatial finite element approximation, derived an error estimate for the time-stepping schemes, and confirmed stability and convergence. The simulations further indicated that the time step $\delta t$  needs to be chosen sufficiently small to prevent excessive numerical diffusion if the backward Euler variant of the time-stepping schemes is used ($\theta=1$).

The numerical investigations revealed the permanence of a critical release threshold separating two qualitatively distinct long-term outcomes: either persistence of the wild mosquito population at a reduced equilibrium level or its eventual local elimination. This threshold phenomenon, initially detected in the ODE-type model \cite{Bliman2019} without spatial heterogeneity, persists in the {spatiotemporal} setting and is strongly influenced not only by the magnitude and timing, but also by the spatial distribution of sterile male releases. The simulations further demonstrate that spatial heterogeneity plays a decisive role: non-uniform releases can create localized differences in suppression intensity whose spatial spread depends on the diffusion rate and the placement of release sources.

While both continuous and periodic impulsive release strategies are effective within the mathematical framework, only impulsive releases are feasible in practice, since field programs operate at discrete time intervals. Consequently, the operational success of SIT-based interventions hinges on selecting the release sizes and frequencies that compensate for mortality and spatial dispersal between releases; otherwise, the system would converge to coexistence rather than elimination.

From a practical standpoint, the results support several operational guidelines: release sizes should exceed the critical threshold with a safety margin; distributed spatial deployment is preferable to highly localized releases; and an initial high-intensity release phase improves suppression speed and robustness. However, evenly distributed releases over large areas may entail higher operational costs, particularly when they require advanced aerial release technologies \cite{Bouyer2020}. If such infrastructure is unavailable, a cost-effective alternative is to position release sources at locations with the highest population densities of the wild mosquitoes identified through spatial monitoring. In this context, adaptive targeting of high-density regions enhances efficiency and reduces sterile male wastage while remaining consistent with the model’s spatial predictions.

{Overall, the model provides a quantitative tool for designing spatially informed SIT programs. By combining rigorous analysis with computational simulations, it bridges theoretical insights and operational decision-making, offering a framework to support evidence-based vector control strategies in spatially heterogeneous environments.}

Future work may extend the present spatial framework in several directions. From a modeling perspective, incorporating first-order (transport) terms in space could help account for directed movement toward mosquito breeding habitats or areas of human aggregation that attract host-seeking female mosquitoes for blood feeding.  Alternative formulations, such as hyperbolic delay models with finite speed of propagation \cite{Hadeler1999, Holmes1993} or wave-type structures for epidemic models as in \cite{Schl2025}, may also provide a more refined description of spatial spread. In addition, introducing open-boundary effects to capture possible inflows and outflows of wild mosquitoes, together with environmental seasonality, parameter uncertainty, and data-assimilation techniques, would enhance realism and further strengthen the predictive and operational value of spatial SIT models.

{\paragraph{Acknowledgement}
S.F. and O.V. acknowledge support by the German Federal Ministry of Economic Affairs and Energy within the project "South American Competence Center of Scientific Computing in Health and Climate" (SaC$^3$) of the German Academic Exchange Service (DAAD), which is a collaboration between 5 South American and 6 German institutions of higher education, including Universidad del Valle and the University of Konstanz. O.E. acknowledges financial support by the Zukunftskolleg of the University of Konstanz by means of a Herz fellowship.}

\bibliographystyle{plain}

\begin{thebibliography}{10}

\bibitem{AbRaSc025}
H.~Abboubakar, R.~Racke, and N.~Schlosser.
\newblock {ODE} and {PDE} models for {COVID}-19, with reinfection and
  vaccination process for {Cameroon} and {Germany}.
\newblock {\em Discrete and Continuous Dynamical Systems -- Series B},
  36:302--334, 2026.

\bibitem{Almeida2022}
L.~Almeida, J.~Estrada, and N.~Vauchelet.
\newblock Wave blocking in a bistable system by local introduction of a
  population: application to sterile insect techniques on mosquito populations.
\newblock {\em Mathematical Modelling of Natural Phenomena}, 17:22, 2022.

\bibitem{Almeida2023}
L.~Almeida, A.~L{\'e}culier, and N.~Vauchelet.
\newblock Analysis of the rolling carpet strategy to eradicate an invasive
  species.
\newblock {\em SIAM Journal on Mathematical Analysis}, 55(1):275--309, 2023.

\bibitem{FEniCS}
M.~Aln{\ae}s, J.~Blechta, J.~Hake, A.~Johansson, B.~Kehlet, A.~Logg,
  C.~Richardson, J.~Ring, M.E. Rognes, and G.N. Wells.
\newblock The {FE}ni{CS} project version 1.5.
\newblock {\em Archive of Numerical Software}, 3(100), 2015.

\bibitem{Anguelov2020p}
R.~Anguelov, Y.~Dumont, and I.~Yatat~Djeumen.
\newblock On the use of traveling waves for pest/vector elimination using the
  sterile insect technique, 2020.
\newblock Preprint. {arXiv}:2010.00861v1.

\bibitem{Anguelov2020}
R.~Anguelov, Y.~Dumont, and I.~Yatat~Djeumen.
\newblock Sustainable vector/pest control using the permanent sterile insect
  technique.
\newblock {\em Mathematical Methods in the Applied Sciences},
  43(18):10391--10412, 2020.

\bibitem{AscherRuuthWetton95}
U.M. Ascher, S.J. Ruuth, and B.T.R. Wetton.
\newblock Implicit-explicit methods for time-dependent partial differential
  equations.
\newblock {\em SIAM Journal on Numerical Analysis}, 32(3):797--823, 1995.

\bibitem{Berg1993}
H.~C. Berg.
\newblock {\em Random walks in biology}.
\newblock Princeton University Press, Princeton NJ, USA, 1993.

\bibitem{Bliman2019}
P.-A. Bliman, D.~Cardona-Salgado, Y.~Dumont, and O.~Vasilieva.
\newblock Implementation of control strategies for sterile insect techniques.
\newblock {\em Mathematical Biosciences}, 314:43--60, 2019.

\bibitem{Bliman2023}
P.-A. Bliman, Y.~Dumont, O.E. Escobar-Lasso, H.J. Martinez-Romero, and
  O.~Vasilieva.
\newblock Sex-structured model of \textit{Wolbachia} invasion and design of
  sex-biased release strategies in \textit{Aedes spp.} mosquitoes populations.
\newblock {\em Applied Mathematical Modelling}, 119:391--412, 2023.

\bibitem{Bourtzis2021}
K.~Bourtzis and M.~Vreysen.
\newblock Sterile {I}nsect {T}echnique ({SIT}) and its applications.
\newblock {\em Insects}, 12(7):638, 2021.

\bibitem{Bouyer2020}
J.~Bouyer, N.~Culbert, A.~Dicko, M.~Pacheco, J.~Virginio, M.~Pedrosa,
  A.~Garziera, L.and~Pinto, A.~Klaptocz, J.~Germann, T.~Wallner,
  G.~Salvador-Herranz, R.~Argiles~Herrero, H.~Yamada, F.~Balestrino, and
  M.~Vreysen.
\newblock Field performance of sterile male mosquitoes released from an
  uncrewed aerial vehicle.
\newblock {\em Science Robotics}, 5(43):eaba6251, 2020.

\bibitem{Capasso1993}
V.~Capasso.
\newblock {\em Mathematical structures of epidemic systems}, volume~97 of {\em
  Lecture Notes in Biomathematics}.
\newblock Springer, Berlin, Germany, 1993.

\bibitem{Cardona2025}
D.~Cardona-Salgado, Y.~Dumont, and O.~Vasilieva.
\newblock Natural population dynamics of {A}sian citrus psyllid,
  \textit{Diaphorina citri}, and its control based on pheromone trapping.
\newblock {\em Mathematical Biosciences}, 390:109540, 2025.

\bibitem{ChOuYa006}
M.~Choulli, E.M. Ouhabaz, and M.~Yamamoto.
\newblock Stable determination of a semilinear termin in a parabolic equation.
\newblock {\em Commun. Pure Appl. Math.}, 5(3):447--462, 2006.

\bibitem{Silva2022}
M.~da~Silva, P.~Lug{\~a}o, F.~Prezoto, and G.~Chapiro.
\newblock Modeling the impact of genetically modified male mosquitoes in the
  spatial population dynamics of \textit{Aedes aegypti}.
\newblock {\em Scientific Reports}, 12(1):9112, 2022.

\bibitem{Araujo2025}
A.~de~Araujo, B.~Boldrini, J.and~Calsavara, and M.~Correa.
\newblock Analysis and numerical approximation of a mathematical model for
  \textit{Aedes aegypti} populations.
\newblock {\em Computers \& Mathematics with Applications}, 180:214--241, 2025.

\bibitem{Dufourd2012}
C.~Dufourd and Y.~Dumont.
\newblock Modeling and simulations of mosquito dispersal. the case of
  \textit{Aedes albopictus}.
\newblock {\em Biomath}, 1(2):ID--1209262, 2012.

\bibitem{Dufourd2013}
C.~Dufourd and Y.~Dumont.
\newblock Impact of environmental factors on mosquito dispersal in the prospect
  of sterile insect technique control.
\newblock {\em Computers \& Mathematics with Applications}, 66(9):1695--1715,
  2013.

\bibitem{Escobar2021}
O.~Escobar-Lasso and O.~Vasilieva.
\newblock A simplified monotone model of \textit{Wolbachia} invasion
  encompassing \textit{Aedes aegypti} mosquitoes.
\newblock {\em Studies in Applied Mathematics}, 146(3):565--585, 2021.

\bibitem{Esteva2005}
L.~Esteva and H.~Yang.
\newblock Mathematical model to assess the control of \textit{Aedes aegypti}
  mosquitoes by the sterile insect technique.
\newblock {\em Mathematical biosciences}, 198(2):132--147, 2005.

\bibitem{FreiHeinlein2023}
S.~Frei and A.~Heinlein.
\newblock Towards parallel time-stepping for the numerical simulation of
  atherosclerotic plaque growth.
\newblock {\em Journal of Computational Physics}, 491:112347, 2023.

\bibitem{Hadeler1999}
K.~P. Hadeler.
\newblock {\em Reaction transport systems in biological modelling}, pages
  95--150.
\newblock Springer Berlin Heidelberg, Berlin, Heidelberg, 1999.

\bibitem{Harrington2005}
L.~C. Harrington, T.~W. Scott, K.~Lerdthusnee, R.~C. Coleman, A.~Costero, G.~G.
  Clark, J.~J. Jones, S.~Kitthawee, P.~Kittayapong, R.~Sithiprasasna, and J.~D.
  Edman.
\newblock Dispersal of the dengue vector \textit{Aedes aegypti} within and
  between rural communities.
\newblock {\em The American Journal of Tropical Medicine and Hygiene},
  72(2):209--220, 2005.

\bibitem{Henry1981}
D.~Henry.
\newblock {\em Geometric theory of semilinear parabolic equations}, volume 840
  of {\em Lecture Notes in Mathematics}.
\newblock Springer-Verlag, Berlin, Germany, 1981.

\bibitem{Holmes1993}
E.E. Holmes.
\newblock Are diffusion models too simple? a comparison with telegraph models
  of invasion.
\newblock {\em The American Naturalist}, 142(5):779--795, 1993.

\bibitem{Huang2017}
M.~Huang, X.~Song, and J.~Li.
\newblock Modelling and analysis of impulsive releases of sterile mosquitoes.
\newblock {\em Journal of biological dynamics}, 11(1):147--171, 2017.

\bibitem{Jiang2014}
W.~Jiang, X.~Li, and X.~Zou.
\newblock On a reaction-diffusion model for sterile insect release method on a
  bounded domain.
\newblock {\em International Journal of Biomathematics}, 7(3):1450030, 2014.

\bibitem{Leculier2023}
A.~L{\'e}culier and N.~Nguyen.
\newblock A control strategy for the sterile insect technique using
  exponentially decreasing releases to avoid the hair-trigger effect.
\newblock {\em Mathematical Modelling of Natural Phenomena}, 18:25, 2023.

\bibitem{Li2012}
X.~Li and X.~Zou.
\newblock On a reaction-diffusion model for sterile insect release method with
  release on the boundary.
\newblock {\em Discrete \& Continuous Dynamical Systems-Series B},
  17(7):2509–--2522, 2012.

\bibitem{Ma2024}
X.~Ma, L.~Cai, and S.~Li.
\newblock Dynamics of interactive wild and sterile mosquitoes in spatially
  heterogenous environment.
\newblock {\em Discrete and Continuous Dynamical Systems-B}, 29(11):4527--4544,
  2024.

\bibitem{Manoranjan1986}
V.~Manoranjan and P.~van~den Driessche.
\newblock On a diffusion model for sterile insect release.
\newblock {\em Mathematical Biosciences}, 79(2):199--208, 1986.

\bibitem{Mora1983}
X.~Mora.
\newblock Semilinear parabolic problems define semiflows on {C}$^k$ spaces.
\newblock {\em Transactions of the American Mathematical Society},
  278(1):21--55, 1983.

\bibitem{Multerer2019}
L.~Multerer, T.~Smith, and N.~Chitnis.
\newblock Modeling the impact of sterile males on an \textit{Aedes aegypti}
  population with optimal control.
\newblock {\em Mathematical biosciences}, 311:91--102, 2019.

\bibitem{Oliva2021}
C.~Oliva, M.~Benedict, C.~Collins, T.~Baldet, R.~Bellini, H.~Bossin, J.~Bouyer,
  V.~Corbel, L.~Facchinelli, F.~Fouque, M.~Geier, A.~Michaelakis, D.~Roiz,
  F.~Simard, C.~Tur, and L.-C. Gouagna.
\newblock Sterile {I}nsect {T}echnique ({SIT}) against \textit{Aedes} species
  mosquitoes: {A} roadmap and good practice framework for designing,
  implementing and evaluating pilot field trials.
\newblock {\em Insects}, 12(3):191, 2021.

\bibitem{Pang2022}
Y.~Pang, S.~Wang, and S.~Liu.
\newblock Dynamics analysis of stage-structured wild and sterile mosquito
  interaction impulsive model.
\newblock {\em Journal of Biological Dynamics}, 16(1):464--479, 2022.

\bibitem{Parshad2011}
R.~Parshad and F.~Agusto.
\newblock Global dynamics of a pde model for \textit{Aedes aegypti} mosquitoes
  incorporating female sexual preference.
\newblock {\em Dynamics of Partial Differential Equations}, 8(4):311--343,,
  2011.

\bibitem{Peng2012}
R.~Peng and X.-Q. Zhao.
\newblock A reaction--diffusion {SIS} epidemic model in a time--periodic
  environment.
\newblock {\em Nonlinearity}, 25(5):1451, 2012.

\bibitem{Ramirez2022}
C.~A. Ramirez-Bernate, H.~J. Martinez-Romero, and D.~M. Erazo-Borja.
\newblock Control strategies in the spatial population dynamics of
  \textit{Aedes aegypti} vector using sterile mosquitoes and insecticides.
\newblock {\em Universitas Scientiarum}, 27(2):206--232, 2022.

\bibitem{RichterBuch}
T.~Richter.
\newblock {\em Fluid-Structure Interactions: Models, Analysis and Finite
  Elements}, volume 118 of {\em Lecture Notes in Computational Science and
  Engineering}.
\newblock Springer Nature, Cham, Switzerland, 2017.

\bibitem{Sarwar2015}
M.~Sarwar.
\newblock The dangers of pesticides associated with public health and
  preventing of the risks.
\newblock {\em International Journal of Bioinformatics and Biomedical
  Engineering}, 1(2):130--136, 2015.

\bibitem{Schl2025}
N.~Schlosser.
\newblock A wave-type model for age-and space-structured epidemics.
\newblock {\em Nonlinear Analysis: Real World Applications}, 91:104612, 2026.

\bibitem{Smoller1983}
J.~Smoller.
\newblock {\em Shock waves and reaction—diffusion equations}, volume 258 of
  {\em Grundlehren der mathematischen Wissenschaften}.
\newblock Springer-Verlag, Berlin, Germany, 1983.

\bibitem{Strugarek2019}
M.~Strugarek, H.~Bossin, and Y.~Dumont.
\newblock On the use of the sterile insect release technique to reduce or
  eliminate mosquito populations.
\newblock {\em Applied Mathematical Modelling}, 68:443--470, 2019.

\end{thebibliography}

\end{document}